\documentclass[a4paper,11pt]{article}
\usepackage{amsmath,amssymb,amsthm,graphicx}
\usepackage[top=80pt, bottom=90pt, left=70pt, right=70pt]{geometry}
\usepackage{comment}
\usepackage{color}
\usepackage{cleveref}
\usepackage{centernot}
\usepackage{url}

\def\qqed{\hfill $\blacksquare$}

\newcommand{\ave}[1]{\ensuremath{\langle#1\rangle} }
\newcommand{\QCP}{{\sc Quartet Compatibility}}
\newcommand{\Q}{\mathcal{Q}}
\newcommand{\A}{\mathcal{A}}
\newcommand{\Dis}{{\sc Displaying}}

\newtheorem{thm}{\bfseries Theorem}[section] 
\newtheorem{lem}[thm]{\bfseries Lemma} 
\newtheorem{prop}[thm]{\bfseries Proposition} 
\newtheorem{cor}[thm]{\bfseries Corollary}

\newtheorem*{cl}{\bfseries Claim}

\theoremstyle{definition}
\newtheorem{exmp}[thm]{\bfseries Example}

\crefname{thm}{Theorem}{Theorems}
\crefname{prop}{Proposition}{Propositions}
\crefname{lem}{Lemma}{Lemmas}
\crefname{exmp}{Example}{Examples}
\crefname{cor}{Corollary}{Corollarys}
\crefname{cl}{Claim}{Claims}
\crefname{remark}{Remark}{Remarks}
\crefname{section}{Section}{Sections}

\usepackage{bm}

\begin{document}
	\title{Reconstructing phylogenetic trees from multipartite quartet systems\thanks{A preliminary version of this paper has appeared in the proceedings of the 29th International Symposium on Algorithms and Computation (ISAAC 2018).}}
	\author{Hiroshi Hirai\thanks{Department of Mathematical Informatics,
			Graduate School of Information Science and Technology,
			The University of Tokyo, Tokyo 113-8656, Japan.
			Email: \texttt{hirai@mist.i.u-tokyo.ac.jp}} \and Yuni Iwamasa\thanks{Department of Communications and Computer Engineering,
Graduate School of Informatics,
Kyoto University, Kyoto 606-8501, Japan.
			Email: \texttt{iwamasa@i.kyoto-u.ac.jp}}}
	\date{\today}
	\maketitle
	
\begin{abstract}
	A phylogenetic tree is a graphical representation of an evolutionary history of taxa in which the leaves correspond to the taxa and the non-leaves correspond to speciations.
	One of important problems in phylogenetic analysis is to assemble  
	a global phylogenetic tree
	from small phylogenetic trees, particularly, quartet trees.
	\QCP\ is the problem of deciding whether there is a phylogenetic tree inducing a given collection of quartet trees,
	and to construct such a phylogenetic tree if it exists.
	It is known that \QCP\ is NP-hard
	and that there are only a few results known for polynomial-time solvable subclasses.
	
	In this paper, we introduce two novel classes of quartet systems, called complete multipartite quartet system and full multipartite quartet system,
	and present polynomial-time algorithms for \QCP\ for these systems.
\end{abstract}
\begin{quote}
\end{quote}
	
\section{Introduction}\label{sec:intro}
A {\it phylogenetic tree} for finite set $[n] := \{1,2, \dots, n\}$ is a tree $T = (V, E)$ such that
the set of leaves of $T$ coincides with $[n]$
and each internal node $V \setminus [n]$ has at least three neighbors.
A phylogenetic tree represents an evolutionary history of taxa
in which the leaves correspond to the taxa and the non-leaves correspond to speciations.
One of important problems in phylogenetic analysis is to assemble  
a global phylogenetic tree  on $[n]$ (called a {\em supertree})
from smaller pieces of phylogenetic trees on possibly overlapping subsets of $[n]$; see~\cite[Section~6]{book/SempleSteel03}.

A {\em quartet tree} (or {\it quartet}) is a smallest nontrivial phylogenetic tree, that is,
it has four leaves (as taxa) and it is not a star.
There are three quartet trees in set $\{a,b,c,d\}$, which are denoted by $ab||cd$, $ac||bd$, and $ad||bc$.
Here $ab||cd$ represents the quartet tree 
such that $a$ and $b$ ($c$ and $d$) are adjacent  to a common node;
see Figure~\ref{fig:quartet}.
\begin{figure}[htb]
	\begin{center}
		\includegraphics[width=8cm]{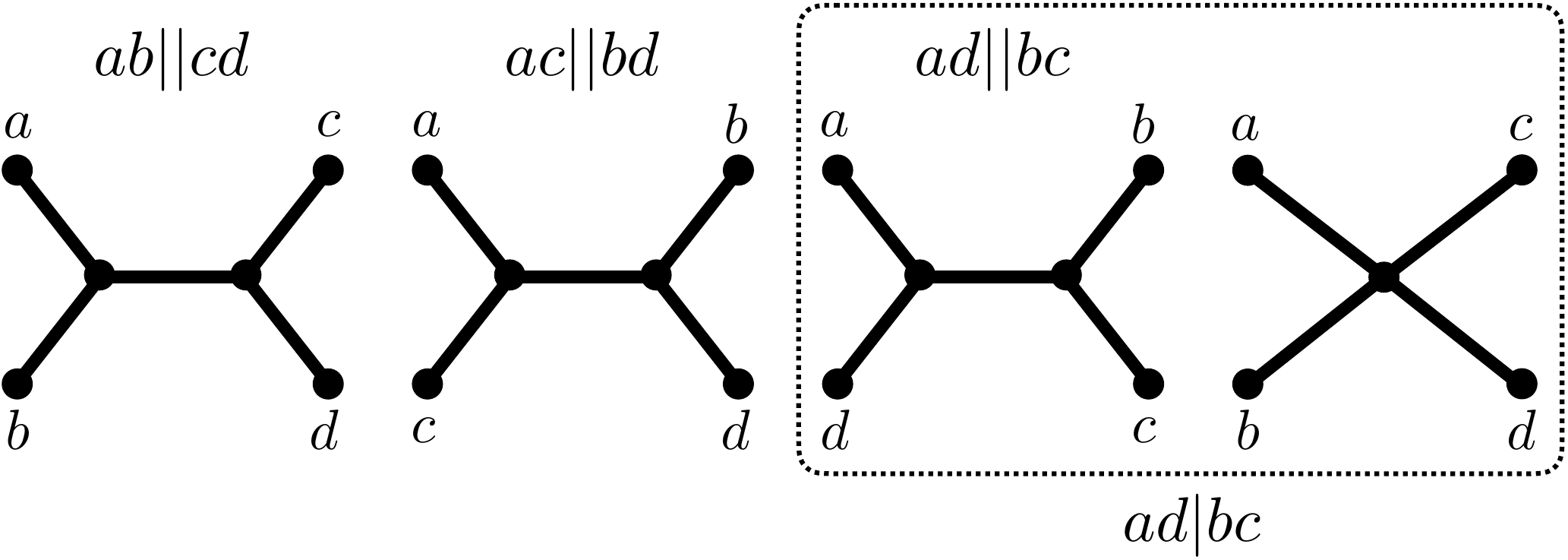}
		\caption{The quartets $ab || cd$, $ac || bd$, and $ad || bc$ represent the first, second, and third phylogenetic trees for $a,b,c,d$ from the left, respectively.
			$ad | bc$, for example, represents one of the two phylogenetic trees in the dotted curve, that is, $ad || bc$ or the star graph with leaves $a,b,c,d$.}
		\label{fig:quartet}
	\end{center}
\end{figure}
Quartet trees are used for representing 
substructures of a (possibly large) phylogenetic tree.
A fundamental problem in phylogenetic analysis is to construct,
from given quartets,
a phylogenetic tree having the quartets as substructures.
To introduce this problem formally, we use some notations and terminologies.
We say that a phylogenetic tree $T$ {\em displays} a quartet $ab||cd$ 
if the simple paths connecting $a,b$ and 
$c,d$ in $T$, respectively, do not meet, i.e., 
$ab||cd$ is the ``restriction'' of $T$ to leaves $a,b,c,d$; see Figure~\ref{fig:phylo tree}.
\begin{figure}[htb]
	\begin{center}
		\includegraphics[width=10cm]{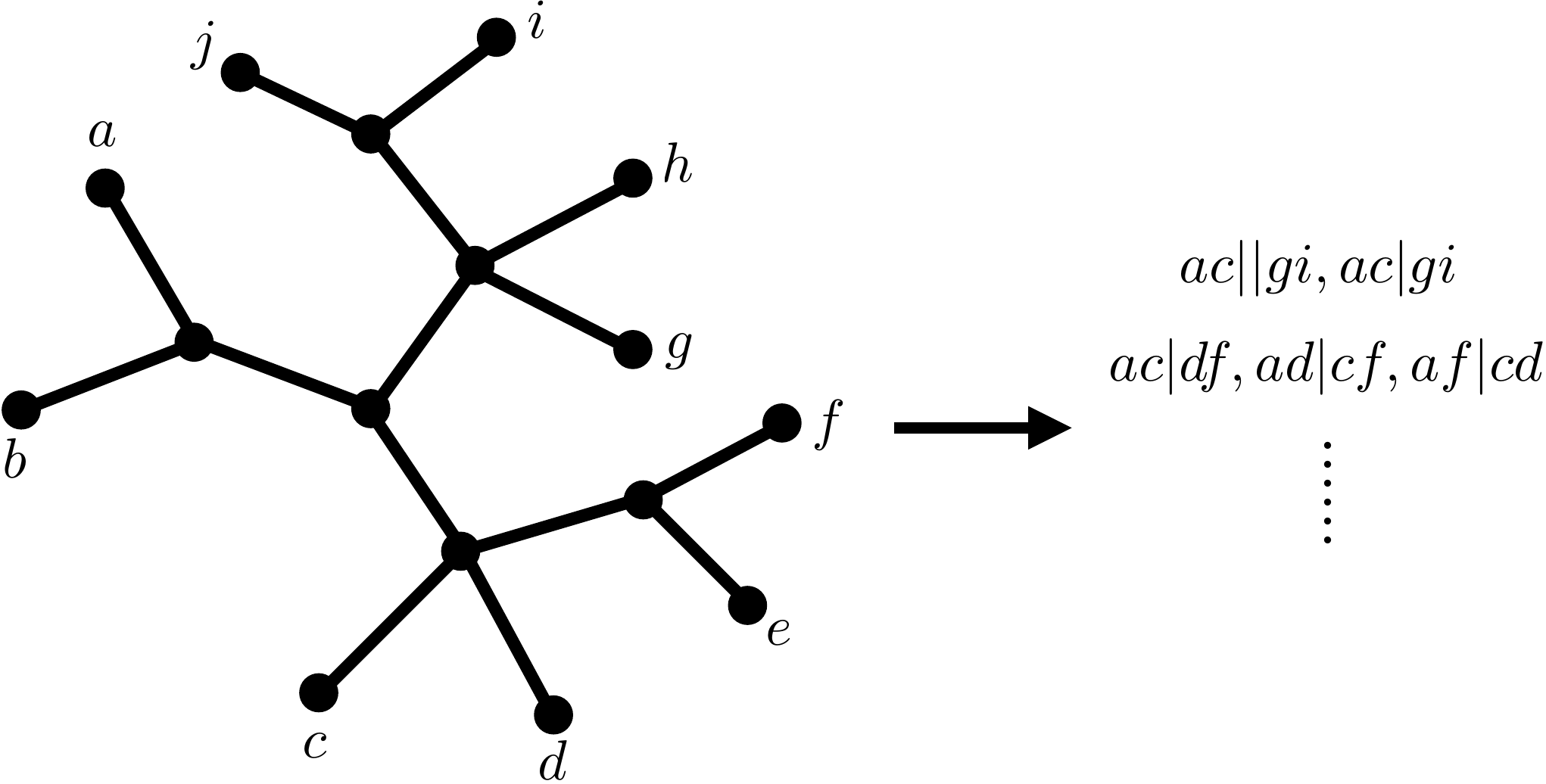}
		\caption{An example of phylogenetic tree $T$ for $\{a,b,c,d,e,f,g,h,i, j\}$.
			$T$ displays, for example, $ac || gi, ac | gi$,
			and $ac|df, ad|cf, af|cd$.
			The quartet trees $ac|df, ad|cf, af|cd$ follow from the same substructure of $T$, the star formed by the restriction of $T$ to $a, c, d, f$.
}
		\label{fig:phylo tree}
	\end{center}
\end{figure}
By a {\em quartet system} on $[n]$ 
we mean a collection of quartet trees whose leaves are subsets of $[n]$.
We say that $T$ {\it displays} a quartet system $\mathcal{Q}$
if $T$ displays all quartet trees in $\mathcal{Q}$.
A quartet system $\mathcal{Q}$ is said to be {\it compatible}
if there exists a phylogenetic tree displaying $\mathcal{Q}$.
Now the problem is formulated as:
\begin{description}
	\item[\underline{\QCP}]
	\item[Given:] A quartet system $\mathcal{Q}$.
	\item[Problem:]
	Determine whether $\mathcal{Q}$ is compatible or not.
	If $\mathcal{Q}$ is compatible, obtain a phylogenetic tree $T$ displaying $\mathcal{Q}$.
\end{description}

\QCP\ has been intensively studied in computational biology as well as theoretical computer science,
particularly, algorithm design and computational complexity.
After a fundamental result by Steel~\cite{JC/S92} on the NP-hardness of \QCP,
there have been a large amount of algorithmic results,
which include efficient heuristics~\cite{AAM/BS95,Master/D86,CAI/ESSW97,PLoSONE/RBR14,JMBE/SH96}
approximation algorithms~\cite{SODA/BBJKLWZ00, ESA/BJKLW99,SICOMP/JKL01},
and parametrized algorithms~\cite{TCS/CLR10,JCSS/GN03}.

In contrast,
there are only a few results on polynomial-time solvable special subclasses:
\begin{itemize}
	\item Colonius and Schulze~\cite{BJMSP/CS81} established a complete characterization to the abstract quaternary relation $N$
	({\em neighbors relation}) obtained from a phylogenetic tree $T$ by: 
	$N(a,b,c,d)$ holds if and only if $T$ displays quartet tree $ab||cd$.
	By using this result, Bandelt and Dress~\cite{AAM/BD86} showed that
	if, for every $4$-element set $\{a,b,c,d\}$ of $[n]$,
	exactly one of $ab||cd$, $ac||bd$, and $ad||cd$ belongs to ${\cal Q}$, 
	then \QCP\ for ${\cal Q}$ can be solved in polynomial time. 
	\item Aho, Sagiv, Szymanski, and Ullman~\cite{SICOMP/ASSU81} devised a polynomial-time algorithm to
	find a {\it rooted} phylogenetic tree displaying the input {\it triple} system.
	By using this result,
	Bryant and Steel~\cite{AAM/BS95} showed that, if all quartets in $\mathcal{Q}$ have a common label,
	then \QCP\ for $\mathcal{Q}$ can be solved in polynomial time.
\end{itemize}
Such results are useful for designing experiments to obtain quartet information from taxa,
and also play key roles in developing supertree methods for (incompatible) phylogenetic trees (e.g.,~\cite{DAM/SS00}).
%

In this paper,
we present two novel tractable classes of quartet systems.
To describe our result, we extend the notions of 
quartets and quartet systems.
In addition to $ab||cd$, 
we consider symbol $ab|cd$ as a quartet, 
which represents 
the quartet tree $ab||cd$ or the star with leaves $a,b,c,d$; see Figure~\ref{fig:quartet}.
This corresponds to the {\em weak neighbors relation} in~\cite{AAM/BD86,BJMSP/CS81}, and enables us to 
capture a degenerate phylogenetic tree in 
which internal nodes may have degree greater than $3$.
In a sense,  $ab|cd$ means a ``possibly degenerate'' quartet tree such that 
the center edge can have zero length,
where edge lengths represent evolutionary distances.
We define that a phylogenetic tree $T$ {\em displays} 
$ab|cd$ if the simple paths connecting $a,b$ and $c,d$ in $T$, 
respectively,  meet at most one node, i.e., 
the restriction of $T$ to $a,b,c,d$ is $ab||cd$ or the star;
see Figure~\ref{fig:phylo tree}.
Then the concepts of quartet systems, 
displaying, compatibility, and \QCP\ are naturally extended.
A quartet system $\mathcal{Q}$ is said to be {\it full} on $[n]$
if, for each distinct $a, b, c, d \in [n]$,
either one of  $ab || cd, ac || bd, ad || bc$ belongs to $\mathcal{Q}$ 
or all $ab | cd, ac | bd, ad | bc$ belong to $\mathcal{Q}$.
The latter situation says that 
any phylogenetic tree displaying $\mathcal{Q}$ should 
induce the star on $a,b,c,d$. 
Actually the above polynomial-time algorithm by Bandelt and Dress~\cite{AAM/BD86} works for full quartet systems.
A quartet system $\mathcal{Q}$ is said to be {\it complete bipartite} relative to bipartition $\{ A,  B\}$ of $[n]$ with $\min \{ |A|,|B| \} \geq 2$
if, for all distinct $a,a' \in A$ and $b,b' \in B$,
exactly one of 
\begin{align}\label{eq:complete bipartite}
ab || a'b', \quad ab' || a'b, \quad aa' | bb'
\end{align}
belongs to $\mathcal{Q}$,
and every quartet in $\mathcal{Q}$ is of the above form~\eqref{eq:complete bipartite}.
Note that every phylogenetic tree displays
exactly one of three quartets in~\eqref{eq:complete bipartite}.

We are ready to introduce the two new classes of quartet systems.
Let $\mathcal{A} := \{A_1, A_2, \dots, A_r\}$ be a partition of $[n]$ with $|A_i| \geq 2$ for all $i \in [r]$.
A quartet system $\mathcal{Q}$ is said to be {\it complete multipartite} relative to $\mathcal{A}$ or {\it complete $\mathcal{A}$-partite}
if $\mathcal{Q}$ is represented as $\bigcup_{1 \leq i < j \leq r} \mathcal{Q}_{ij}$
for complete bipartite quartet systems $\mathcal{Q}_{ij}$ on $A_{i} \cup A_{j}$ 
with bipartition $\{A_i, A_j\}$. 
A quartet system $\mathcal{Q}$ is said to be {\it full multipartite} relative to $\mathcal{A}$ or {\it full $\mathcal{A}$-partite}
if $\mathcal{Q}$ is represented as $\mathcal{Q}_0 \cup \mathcal{Q}_1 \cup \cdots \cup \mathcal{Q}_r$,
where $\mathcal{Q}_0$ is a complete $\mathcal{A}$-partite quartet system and
$\mathcal{Q}_i$ is a full quartet system on $A_i$ for each $i \in [r]$.
%
Our main result is:
\begin{thm}\label{thm:main}
	If the input quartet system $\mathcal{Q}$ is complete $\mathcal{A}$-partite 
	or full $\mathcal{A}$-partite,
	then
	\QCP\ can be solved in $O(|\mathcal{A}|n^4)$ time.
\end{thm}
The result for full $\mathcal{A}$-partite quartet systems 
extends the above polynomial-time solvability for full quartet systems by~\cite{AAM/BD86}.
Also this result may have an insight on supertree construction from
phylogenetic trees on disjoint groups of taxa.
Namely, the existence of a supertree of these phylogenetic trees
can be decided in polynomial time if the interrelation among these groups
are described as a multipartite quartet system.

\paragraph{Organization.}
This paper is organized as follows.
\QCP\ can be viewed as a problem of finding an appropriate laminar family.
We first introduce a displaying concept for an arbitrary family of subsets,
and then divide \QCP\ into two subproblems.
The first problem is to find a family displaying the input quartet system,
and the second problem is to transform the family into a desired laminar family.
For the second,
we utilize the {\it laminarization algorithm}
developed by Hirai, Iwamasa, Murota, and \v{Z}ivn\'y~\cite{TALG/HIMZ19}
obtained in a different context.
In \cref{sec:complete,sec:full},
we show the results for complete and full multipartite quartet systems,
respectively.
In \cref{sec:remark},
we suggest a possible situation in which our results are applicable.

\paragraph{Preliminaries.}
A family $\mathcal{L} \subseteq 2^{[n]}$ is said to be {\it laminar}
if $X \subseteq Y$, $X \supseteq Y$, or $X \cap Y = \emptyset$ holds for all $X, Y \in \mathcal{L}$.
A phylogenetic tree can be encoded into a laminar family as follows.
Let $T = (V, E)$ be a phylogenetic tree for $[n]$.
We say that an edge in $E$ is {\it internal}
if it is not incident to a leaf.
By deleting an internal edge $e \in E$,
the tree $T$ is separated into two connected components,
which induce a bipartition of $[n]$.
We denote by $\{X_e, Y_e\}$ this bipartition.
By choosing either $X_e$ or $Y_e$ appropriately for each internal edge $e \in E$,
we can construct a laminar family $\mathcal{L}$ on $[n]$ with $\min\{|X|, |[n] \setminus X|\} \geq 2$ for all $X \in\mathcal{L}$.
Conversely, let $\mathcal{L}$ on $[n]$ be a laminar family with $\min\{|X|, |[n] \setminus X|\} \geq 2$ for all $X \in\mathcal{L}$.
Then we construct the set $\hat{\mathcal{L}} := \{ \{ X, [n] \setminus X \} \mid X \in \mathcal{L} \}$ of bipartitions from $\mathcal{L}$.
It is known~\cite{MAHS/B71} that, for such $\hat{\mathcal{L}}$, there uniquely exists a phylogenetic tree that induces $\hat{\mathcal{L}}$.

\section{Complete multipartite quartet system}\label{sec:complete}
\subsection{{\sc Displaying} and {\sc Laminarization}}\label{subsec:complete preliminary}
In this subsection,
we explain that \QCP\ for complete multipartite quartet systems can be divided into two subproblems named as {\sc Displaying} and {\sc Laminarization}.
Let $\mathcal{A} := \{A_1, A_2, \dots, A_r\}$ be a partition of $[n]$ with $|A_i| \geq 2$ for all $i \in [r]$,
and $\mathcal{Q}$ be a complete $\mathcal{A}$-partite quartet system.
We say that a family $\mathcal{F} \subseteq 2^{[n]}$ {\it displays} $\mathcal{Q}$
if, for all distinct $i, j \in [r]$, $a, a' \in A_i$, and $b, b' \in A_j$,
the following (i) and (ii) are equivalent:
\begin{itemize}
	\item[(i)] $ab || a'b'$ belongs to $\mathcal{Q}$.
	\item[(ii)] There is $X \in \mathcal{F}$ satisfying $a, b \in X \not\ni a', b'$ or $a, b \not\in X \ni a', b'$.
\end{itemize}
Every family displays either exactly one complete $\mathcal{A}$-partite quartet system or no complete $\mathcal{A}$-partite quartet system.
Indeed,
a family $\mathcal{F}$ uniquely determine the set $\mathcal{Q}(\mathcal{F})$ of quartet trees of the form $ab || a'b'$
by the above correspondence.
The family $\mathcal{F}$ does not display any complete $\mathcal{A}$-partite quartet system
if and only if
$\mathcal{Q}(\mathcal{F})$ contains both $ab || a'b'$ and $ab' || a'b$
for some distinct $i, j \in [r]$, $a, a' \in A_i$, and $b, b' \in A_j$.
In particular, a laminar family $\mathcal{L}$ always displays exactly one complete $\mathcal{A}$-partite quartet system $\mathcal{Q}$.

We can easily see that a complete $\mathcal{A}$-partite system $\mathcal{Q}$ is compatible if and only if there exists a laminar family $\mathcal{L}$ displaying $\mathcal{Q}$,
and that in the compatible case,
$\mathcal{Q}$ is displayed by
the phylogenetic tree corresponding to $\mathcal{L}$.
Hence \QCP\ for a complete $\mathcal{A}$-partite quartet system $\mathcal{Q}$ can be viewed as the problem of finding a laminar family $\mathcal{L}$ displaying $\mathcal{Q}$ if it exists.

It can happen that different families may display the same complete $\mathcal{A}$-partite quartet system.
To cope with such complications,
we define notions on subsets of $[n]$.
For $X \subseteq [n]$,
define
\begin{align}\label{eq:ave}
\ave{X} := \bigcup \{A_i \in \mathcal{A} \mid \emptyset \neq X \cap A_i \neq A_i \}.
\end{align}
A set $X \subseteq [n]$ is called an {\it $\mathcal{A}$-cut}
if $\ave{X} \supseteq A_i \cup A_j$ for some distinct $i,j \in [r]$,
i.e., $\emptyset \neq  X \cap A_i \neq A_i$ holds for at least two $i \in [r]$.
We define an equivalence relation $\sim$ on sets $X,Y \subseteq [n]$ by:
\begin{align}\label{eq:sim}
X \sim Y \iff \{ \ave{X} \cap X, \ave{X} \setminus X \} = \{ \ave{Y} \cap Y, \ave{Y} \setminus Y \}.
\end{align}
By easy observations,
one can see that $X$ is an $\mathcal{A}$-cut
if and only if $\{X\}$ and $\{\emptyset\}$ display different complete $\mathcal{A}$-partite quartet systems,
and
that for $\mathcal{A}$-cuts $X$ and $Y$, $X \sim Y$
if and only if $\{X\}$ and $\{ Y \}$ display the same complete $\mathcal{A}$-partite quartet system.
We consider only $\mathcal{A}$-cuts
if the input quartet system $\mathcal{Q}$ is complete $\mathcal{A}$-partite.
Indeed, let $\mathcal{F}$ be a family and $\mathcal{F}'$ the $\mathcal{A}$-cut family in $\mathcal{F}$.
Then both $\mathcal{F}$ and $\mathcal{F}'$ display the same complete $\mathcal{A}$-partite quartet system.
Let $[X] := \{ Y \subseteq [n] \mid X \sim Y \}$ for $X \subseteq [n]$.
The equivalence relation is naturally extended to families $\mathcal{F}, \mathcal{G}$ by:
$\mathcal{F} \sim \mathcal{G} \Leftrightarrow \mathcal{F} / {\sim} = \mathcal{G} / {\sim}$,
where $\mathcal{F} / {\sim} := \{ [X] \mid X \in \mathcal{F} \}$.
We can observe that, for $\mathcal{A}$-cut families $\mathcal{F}$ and $\mathcal{G}$,
if $\mathcal{F} \sim \mathcal{G}$
then both $\mathcal{F}$ and $\mathcal{G}$ display the same complete $\mathcal{A}$-partite quartet system.
A family $\mathcal{F}$ is said to be {\it laminarizable}
if there is a laminar family $\mathcal{L}$ with $\mathcal{F} \sim \mathcal{L}$.

By the above arguments,
\QCP\ for a complete $\mathcal{A}$-partite quartet system $\mathcal{Q}$ can be divided into the following two subproblems:
\begin{description}
	\item[\underline{\sc Displaying}]
	\item[Given:] A complete $\mathcal{A}$-partite quartet system $\mathcal{Q}$.
	\item[Problem:] Either detect the incompatibility of $\mathcal{Q}$,
	or obtain some $\mathcal{A}$-cut family $\mathcal{F}$ displaying $\mathcal{Q}$.
	In addition, if $\mathcal{Q}$ is compatible,
	then $\mathcal{F}$ should be laminarizable.
	\item[\underline{\sc Laminarization}]
	\item[Given:] An $\mathcal{A}$-cut family $\mathcal{F}$.
	\item[Problem:] Determine whether
	$\mathcal{F}$ is laminarizable or not.
	If $\mathcal{F}$ is laminarizable,
	obtain a laminar $\mathcal{A}$-cut family $\mathcal{L}$ with 
	$\mathcal{L} \sim \mathcal{F}$.
\end{description}
Here, in {\sc Laminarization}, we assume that no distinct $X, Y$ with $X \sim Y$ are contained in $\mathcal{F}$,
i.e., $|\mathcal{F}| = |\mathcal{F} / {\sim}|$.

\QCP\ for complete multipartite quartet systems can be solved as follows.
\begin{itemize}
	\item Suppose that $\mathcal{Q}$ is compatible.
	First, by solving {\sc Displaying}, we obtain a laminarizable $\mathcal{A}$-cut family $\mathcal{F}$ displaying $\mathcal{Q}$.
	Then, by solving {\sc Laminarization} for $\mathcal{F}$, we obtain a laminar $\mathcal{A}$-cut family $\mathcal{L}$ with $\mathcal{L} \sim \mathcal{F}$.
	Since $\mathcal{L} \sim \mathcal{F}$,
	$\mathcal{L}$ also displays $\mathcal{Q}$.
	\item Suppose that $\mathcal{Q}$ is not compatible.
	By solving {\sc Displaying},
	we can detect the incompatibility of $\mathcal{Q}$ or we obtain some $\mathcal{A}$-cut family $\mathcal{F}$ displaying $\mathcal{Q}$.
	In the former case, we are done.
	In the latter case,
	by solving {\sc Laminarization} for $\mathcal{F}$,
	we can detect the non-laminarizability of $\mathcal{F}$,
	which implies the incompatibility of $\mathcal{Q}$.
\end{itemize}

\begin{exmp}\label{ex:reconstruction}
	Let $\mathcal{A} := \{ \{a,b,c\}, \{d,e\}, \{f,g\}, \{h,i\} \}$.
	We illustrate how to solve {\sc Quartet Compatibility} for a complete $\mathcal{A}$-partite quartet system $\mathcal{Q} := \bigcup_{1 \leq i < j \leq 4} \mathcal{Q}_{ij}$,
	where
	\begin{alignat*}{2}
	\mathcal{Q}_{12} &:= \{ ab | de, ad || ce, bd || ce \}, & \qquad \mathcal{Q}_{13} &:= \{ ag || bf, ag || cf, bg || cf \},\\ \mathcal{Q}_{14} &:= \{ ab | hi, ac | hi, bc | hi \}, & \mathcal{Q}_{23} &:= \{ dg || ef \},\\
	\mathcal{Q}_{24} &:= \{ di || eh \}, & \mathcal{Q}_{34} &:= \{ fi || gh \};
	\end{alignat*}
	see also Figure~\ref{fig:example}.
	
	We first solve {\sc Displaying} for $\mathcal{Q}$,
	and obtain an $\mathcal{A}$-cut family $\mathcal{F} := \{ \{a,b,d,g\}, \{a,g\}, \{d,i\}, \{g,h\} \}$
	that displays $\mathcal{Q}$.
	Note that $\mathcal{F}$ is not laminar.
	Then, by solving {\sc Laminarization} for $\mathcal{F}$,
	we obtain a laminar $\mathcal{A}$-cut family $\mathcal{L} := \{ \{a,b,d,g\}, \{a,g\}, \{ a,b,c,d,f,g,i \}, \{f,i\} \}$ with $\mathcal{L} \sim \mathcal{F}$.
	Indeed, we have $\{ d,i \} \sim\{ a,b,c,d,f,g,i \}$ and $\{g,h\} \sim \{e,i\}$.
	This implies that $\mathcal{Q}$ is compatible
	and is displayed by the phylogenetic tree corresponding to $\mathcal{L}$.
	\qqed
\end{exmp}
\begin{figure}[p]
	\begin{center}
		\includegraphics[width=15cm]{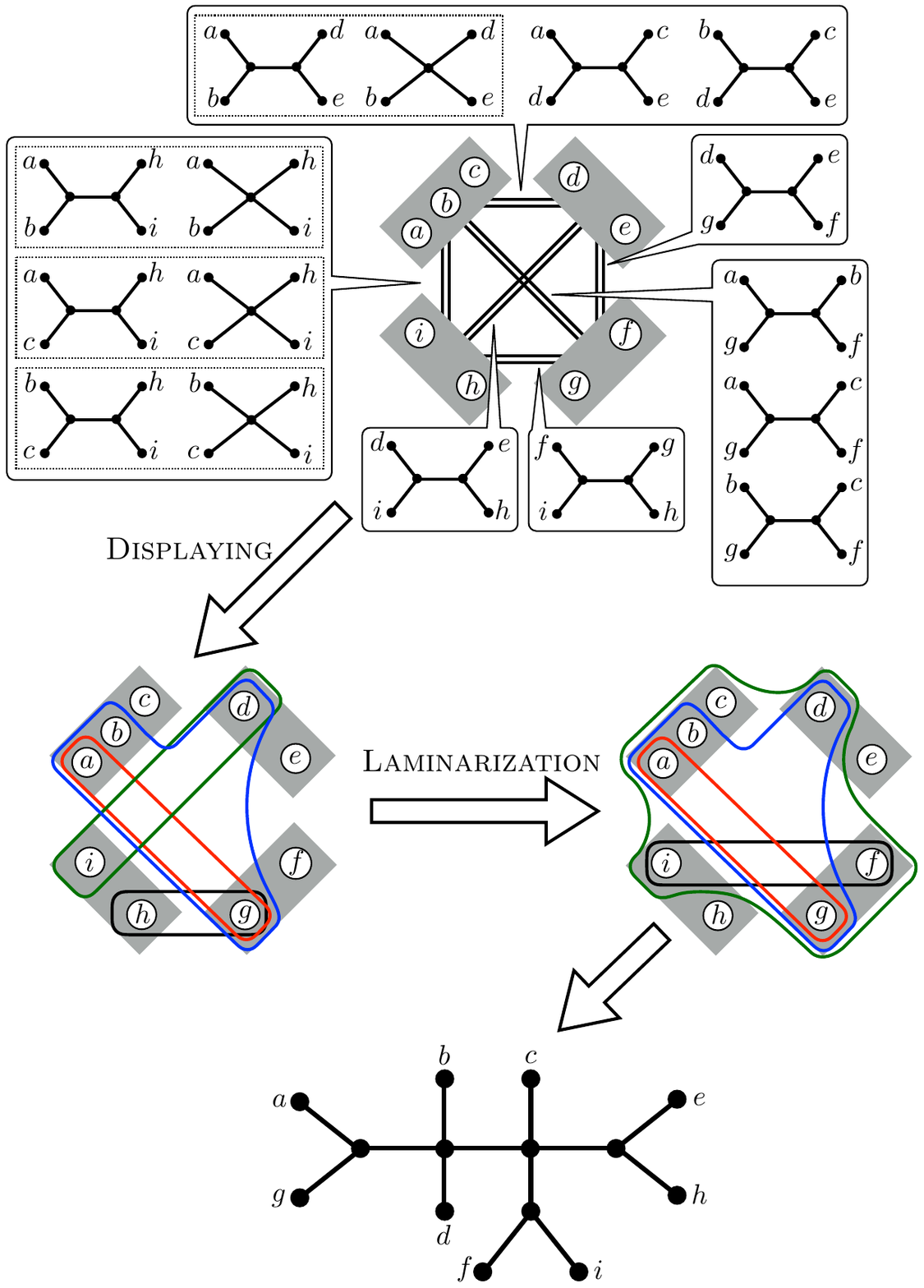}
		\caption{The outline of our algorithm for reconstructing a phylogenetic tree from a complete $\mathcal{A}$-partite system $\mathcal{Q}$ defined in \cref{ex:reconstruction}.}
		\label{fig:example}
	\end{center}
\end{figure}

In~\cite{TALG/HIMZ19},
the authors presented an $O(n^4)$-time algorithm for {\sc Laminarization}.
\begin{thm}[{\cite{TALG/HIMZ19}}]\label{thm:laminarization}
	{\sc Laminarization} can be solved in $O(n^4)$ time.
\end{thm}
The main part of this paper is to show the polynomial-time solvability
of {\sc Displaying},
which implies \cref{thm:main} for complete $\mathcal{A}$-partite quartet systems.
\begin{thm}\label{thm:displaying}
	{\sc Displaying} can be solved in $O(|\mathcal{A}|n^4)$ time.
\end{thm}

An outline of our algorithm for {\sc Displaying} is the following.
\begin{itemize}
	\item
	The building block is an $O(|\mathcal{Q}|)$-time algorithm solving \QCP\ 
	(not {\sc Displaying})
	for a bipartite quartet system $\mathcal{Q}$ on $A \cup B$,
	which is described in \cref{subsec:bipartite algo}.
	The algorithm consists of two steps.
	We first fix $a_0$ in $A$,
	and, for each $a \in A \setminus \{ a_0 \}$,
	construct a laminar family displaying the restriction of $\mathcal{Q}$ to $\{a_0,a\} \cup B$ (\cref{subsubsec:r=2 |A|=2}).
	Next,
	we combine the resulting $|A|-1$ laminar families
	and obtain a laminar family displaying $\mathcal{Q}$ if $\mathcal{Q}$ is compatible (\cref{subsubsec:r=2 general}).
	
	\item
	Then the displaying algorithm for
	a general complete $\A$-partite quartet system $\Q = \bigcup_{1 \leq i < j \leq r} \Q_{ij}$ is
	described in \cref{subsec:complete algo} as:
	For each $i,j$, 
	consider the bipartite quartet system $\Q_{ij}$ on $A_i \cup A_j$,
	and
	obtain a laminar family displaying $\Q_{ij}$ by the above algorithm.
	We combine the resulting laminar families to obtain a solution of \Dis,
	which is guaranteed to be laminarizable if $\Q$ is compatible.
\end{itemize}
For proving the correctness of the above combining procedures,
we reveal a uniqueness property of
a laminarizable $\A$-cut family displaying
a compatible quartet system,
which seems interesting in its own right:
For a bipartite quartet system $\Q_{ij}$,
a laminarizable
family displaying $\Q_{ij}$ is uniquely determined up
to the equivalence relation $\sim$ (\cref{prop:r=2 unique}).
Also for a multipartite quartet system $\Q$,
a minimal laminarizable family displaying $\Q$ is uniquely
determined up to $\sim$ (\cref{prop:r>=3 unique}).
In fact, the above displaying algorithm
outputs this minimal laminarizable family.

\subsection{Algorithm for complete bipartite quartet system}\label{subsec:bipartite algo}
We first construct a polynomial-time algorithm for \QCP\ for complete bipartite quartet systems.
In the following,
$\mathcal{A}$ is a bipartition of $[n]$ represented as $\{A, B\}$ with $\min \{ |A|, |B| \} \geq 2$.
In this case, $X$ is an $\mathcal{A}$-cut if and only if $\emptyset \neq X \cap A \neq A$ and $\emptyset \neq X \cap B \neq B$,
and for $\mathcal{A}$-cuts $X$ and $Y$, $X \sim Y$ if and only if $X = Y$ or $X = [n] \setminus Y$.

For a compatible bipartite quartet system $\mathcal{Q}$,
there is a laminar $\mathcal{A}$-cut family $\mathcal{F}$ displaying $\mathcal{Q}$ such that there is no $X \in \mathcal{F}$ with given $a \in X$.
The following proposition implies that such $\mathcal{F}$ is unique.
\begin{prop}\label{prop:r=2 unique}
	Suppose that a complete bipartite quartet system $\mathcal{Q}$ is compatible.
	Then a laminarizable $\mathcal{A}$-cut family $\mathcal{F}$ displaying $\mathcal{Q}$ is uniquely determined up to $\sim$.
\end{prop}
\begin{proof}
	Let $\mathcal{F}$ and $\mathcal{F}'$ be laminar $\mathcal{A}$-cut families with $\mathcal{F} \not\sim \mathcal{F}'$,
	and $\mathcal{Q}$ and $\mathcal{Q}'$ be the complete bipartite quartet systems displayed by $\mathcal{F}$ and $\mathcal{F}'$, respectively.
	It suffices to show $\mathcal{Q} \neq \mathcal{Q}'$.
	If $\mathcal{F} = \emptyset$ or $\mathcal{F}' = \emptyset$,
	then obviously $\mathcal{Q} \neq \mathcal{Q}'$.
	Hence, in the following, we can assume that $\mathcal{F} \neq \emptyset \neq \mathcal{F}'$ and some $Z \in \mathcal{F}$ satisfies $Z \not\sim X'$ for all $X' \in \mathcal{F}'$.
	
	Let $a \in Z \cap A$ and $b \in Z \cap B$ be elements in $[n]$ such that $Z$ is the minimal member in $\mathcal{F}$ containing $a$ and $b$.
	Suppose that, for all $X' \in \mathcal{F}'$, we have $a \in X' \not\ni b$ or $a \not\in X' \ni b$.
	Then, for $a' \in A \setminus Z$ and $b' \in B \setminus Z$,
	it holds that $\mathcal{Q} \ni ab || a'b' \not\in \mathcal{Q}'$,
	as required.
	Hence, without loss of generality,
	we can assume that there is a member in $\mathcal{F}'$ containing both $a$ and $b$.
	Let $Z'$ be the minimal member in $\mathcal{F}'$ containing $a$ and $b$;
	such $Z'$ is uniquely determined by the laminarity.
	Then following holds.
	\begin{cl}
		For $a' \in A$ and $b' \in B$,
		it holds
		$ab || a'b' \not\in \mathcal{Q}'$ if exactly one of $a'$ and $b'$ belongs to $Z'$.
	\end{cl}
	\begin{proof}[Proof of Claim.]
		If $ab || a'b' \in \mathcal{Q}'$ for some $a' \in Z' \cap A$ and $b' \in B \setminus Z'$,
		then there is $X' \in \mathcal{F}'$ such that $a, b \in X' \not\ni a', b'$ or $a, b \not\in X' \ni a', b'$.
		The former case contradicts the minimality of $Z'$,
		and the latter case contradicts the laminarity of $\mathcal{F}'$.
	\end{proof}
	
	In the following,
	we show that there is a quartet tree $ab || a'b'$ in the symmetric difference between $\mathcal{Q}$ and $\mathcal{Q}'$
	for each of the four cases of the pair $(Z, Z')$.
	
	(i) $Z \subseteq Z'$.
	Without loss of generality,
	we assume $(Z' \setminus Z) \cap A \neq \emptyset$.
	For $a' \in (Z' \setminus Z) \cap A$ and $b' \in B \setminus Z'$,
	it holds that $\mathcal{Q} \ni ab || a'b' \not\in \mathcal{Q}'$
	by Claim,
	as required.
		
		(ii) $Z \supseteq Z'$.
		We have $Z' \not\in \mathcal{F}$ by the minimality of $Z$ and $[n] \setminus Z' \not\in \mathcal{F}$
		by the laminarity of $\mathcal{F}$.
		Hence $Z'$ satisfies $Z' \not\sim X$ for all $X \in \mathcal{F}$.
		By changing the role of $Z$ and $Z'$,
		reducing to (i).
		
		(iii) $Z \not\subseteq Z'$, $Z \not\supseteq Z'$, and $Z \cup Z' \neq [n]$.
		We can take a pair $(a', b')$ such that [$a' \in A \cap (Z' \setminus Z)$ and $b' \in B \setminus (Z \cup Z')$] or [$a' \in A \setminus (Z \cup Z')$ and $b' \in B \cap (Z' \setminus Z)$].
		Indeed, suppose that $A \cap (Z' \setminus Z) = \emptyset$ or $B \subseteq Z \cup Z'$.
		Then $A \setminus (Z \cup Z') \neq \emptyset$ and $B \cap (Z' \setminus Z) \neq \emptyset$ hold.
		For such a pair $(a', b')$,
		it holds that $\mathcal{Q} \ni ab || a'b' \not\in \mathcal{Q}'$ by Claim,
		as required.
		
		(iv) $Z \not\subseteq Z'$, $Z \not\supseteq Z'$, and $Z \cup Z' = [n]$.
		By replacing $X$ with $[n] \setminus X$ in $\mathcal{F}$ appropriately,
		we can redefine $\mathcal{F}$ as a laminar family containing $[n] \setminus Z$.
		Retake $a \in ([n] \setminus Z) \cap A$ and $b \in ([n] \setminus Z) \cap B$ as elements
		such that $[n] \setminus Z$ is a minimal subset of $[n]$ in $\mathcal{F}$ containing $a$ and $b$.
		Let $Z''$ be a minimal element in $\mathcal{F}'$ containing the new $a$ and $b$.
		Since $Z' \supseteq [n] \setminus Z$ and $Z' \supseteq Z''$,
		we have $([n] \setminus Z) \cup Z'' \neq [n]$.
		Thus
		(iv) can reduce to (i), (ii), or (iii) for the pair $([n] \setminus Z, Z'')$.
	\end{proof}
	

\subsubsection{Case of $|A| = 2$ or $|B| = 2$}\label{subsubsec:r=2 |A|=2}
We consider the case of $|A| = 2$ or $|B| = 2$.
We may assume $A = \{a_0, a\}$ with $a_0 \neq a$.
In the following, we abbreviate $\{a_0, a\}$ as $a_0a$.
For $\mathcal{F} \subseteq 2^{[n]}$ and $X \subseteq [n]$,
we denote $\{ F \cup X \mid F \in \mathcal{F} \}$ by $\mathcal{F} \sqcup X$.
For $C \subseteq A$ and $D \subseteq B$, we denote by $\mathcal{Q}_{C,D}$ the set of quartet trees for $c, c' \in C$ and $d, d' \in D$ in $\mathcal{Q}$.

We first explain the idea behind our algorithm (Algorithm~1).
Assume that a complete $\{ a_0a, B \}$-partite quartet system $\mathcal{Q}$ is compatible.
By \cref{prop:r=2 unique},
there uniquely exists a laminar $\{ a_0a, B \}$-cut family $\mathcal{F}$ displaying $\mathcal{Q}$
such that no $X \in \mathcal{F}$ contains $a_0$.
This implies that all $X \in \mathcal{F}$ contains $a$ since $\mathcal{F}$ is an $\{ a_0a, B \}$-cut family.
Hence, by the laminarity,
$\mathcal{F}$ is a chain $\{B_1, B_2, \dots, B_m\} \sqcup \{a\}$ with $\emptyset =: B_0 \subsetneq B_1 \subsetneq B_2 \subsetneq \cdots \subsetneq B_m \subsetneq B_{m+1} := B$.

Choose an arbitrary $b \in B$. 
Consider the index $k \in [m+1]$ 
such that $b \in B_k$ and $b \not\in B_{k-1}$.
Partition $B$ into three sets
$B^- := B_{k-1}$, $B^= := B_k \setminus B_{k-1}$, and $B^+ := B \setminus B_k$.
The tripartition $\{ B^-, B^=, B^+\}$ can be determined
by checking quartets in ${\cal Q}$ having leaves $a_0, a,b$,
where recall~\eqref{eq:complete bipartite} for the definition of a complete bipartite system:
\begin{align}
b' \in B^-  &\iff  a_0b || ab' \in \mathcal{Q},\label{eq:B^-} \\
b' \in B^=  &\iff   b' = b\ {\rm or}\ a_0a | bb' \in \mathcal{Q},\label{eq:B^=} \\
b' \in B^+  &\iff  a_0b' || ab \in \mathcal{Q}.\label{eq:B^+}
\end{align}
Observe that $\mathcal{Q}_{a_0a, B^-}$ is displayed by 
${\cal F}^- := \{ B_1, \dots, B_{k-2} \} \sqcup \{a\}$ and that $\mathcal{Q}_{a_0a, B^+}$ is displayed by 
${\cal F}^+ := \{ B_{k+1} \setminus B_{k}, \dots, B_m \setminus B_{k} \} \sqcup \{a\}$.
After determining $B^-, B^=, B^+$,   
we can apply recursively the same procedure to $\mathcal{Q}_{a_0a, B^-}$ 
and $\mathcal{Q}_{a_0a, B^+}$, and
obtain ${\cal F}^-$ and
${\cal F}^+$.
Combining them with $B_k = B^= \cup B^-$ and $B_{k-1} = B^-$, we obtain $\mathcal{F} = \{B_1, B_2, \dots, B_m\} \sqcup \{a\}$ as required.

The formal description of Algorithm~1 is the following:
\begin{description}
	\item[Algorithm~1 (for complete $\{ a_0a, B \}$-partite quartet system with pivot $a$):]
	\item[Input:] A complete $\{ a_0a, B \}$-partite quartet system $\mathcal{Q}$.
	\item[Output:] Either detect the incompatibility of $\mathcal{Q}$,
	or obtain the (unique) laminar $\{ a_0a, B \}$-cut family $\mathcal{F}$ displaying $\mathcal{Q}$ such that no $X \in \mathcal{F}$ contains $a_0$.
	\item[Step 1:]
	Define $\mathcal{F} := \emptyset$.
	If $\mathcal{Q} = \emptyset$, that is, $|B|$ is at most one,
	then output $\mathcal{F}$ and stop.
	\item[Step 2:]
	Choose an arbitrary $b \in B$.
	Define
	$B^-$,
	$B^=$,
	and $B^+$ as~\eqref{eq:B^-},~\eqref{eq:B^=}, and~\eqref{eq:B^+}, respectively.
	If $B^- \neq \emptyset$,
	then update $\mathcal{F} \leftarrow \mathcal{F} \cup \{B^- \cup a\}$.
	If $B^- \cup B^= \neq B$,
	then update $\mathcal{F} \leftarrow \mathcal{F} \cup \{B^- \cup B^= \cup a\}$.
	
	\item[Step 3:]
	If Algorithm~1 for $\mathcal{Q}_{a_0a, B^+}$ with pivot $a$ detects the incompatibility of $\mathcal{Q}_{a_0a, B^+}$ or Algorithm~1 for $\mathcal{Q}_{a_0a, B^-}$ with pivot $a$ detects the incompatibility of $\mathcal{Q}_{a_0a, B^-}$,
	then output ``$\mathcal{Q}$ is not compatible'' and stop.
	Otherwise, let $\mathcal{F}^+$ and $\mathcal{F}^-$ be the output families of Algorithm~1 for $\mathcal{Q}_{a_0a, B^+}$ and for $\mathcal{Q}_{a_0a, B^-}$, respectively.
	Update
	\begin{align*}
	\mathcal{F} \leftarrow \mathcal{F} \cup \mathcal{F}^- \cup \left(\mathcal{F}^+ \sqcup (B^- \cup B^=)\right).
	\end{align*}
	\item[Step 4:] If $\mathcal{F}$ displays $\mathcal{Q}$,
	then output $\mathcal{F}$.
	Otherwise, output ``$\mathcal{Q}$ is not compatible.''
	\qqed
\end{description}

\begin{prop}\label{prop:algo 1}
	Algorithm~{\rm 1} solves \QCP\ for a complete $\{ a_0a, B \}$-partite quartet system $\mathcal{Q}$ in $O(|\mathcal{Q}|)$ time.
\end{prop}
	\begin{proof}
		(Validity).
		It suffices to show that
		if $\mathcal{Q}$ is compatible then Algorithm~1 outputs an laminar $\mathcal{A}$-cut family $\mathcal{F}$ displaying $\mathcal{Q}$,
		and it follows from the argument before the formal description of Algorithm~1.

		(Complexity).
		We can easily see that, in Steps~2 and~3, Algorithm~1 checks a quartet tree for $a_0, a_1, b, b'$ at most one time for each distinct $b, b' \in B$.
		It takes $|Q|$ time to check whether $\mathcal{F}$ displays $\mathcal{Q}$ in Step~4.
		Thus the running-time is $O(|\mathcal{Q}|)$.
	\end{proof}

\subsubsection{General case}\label{subsubsec:r=2 general}
We consider general complete bipartite quartet systems; $\mathcal{A}$ is a bipartition $\{A, B\}$ of $[n]$.
As in \cref{subsubsec:r=2 |A|=2}, we first explain the idea behind our algorithm (Algorithm~2).
Assume that a complete $\mathcal{A}$-partite quartet system $\mathcal{Q}$ is compatible.
By \cref{prop:r=2 unique},
there uniquely exists a laminar $\mathcal{A}$-cut family $\mathcal{F}$ displaying $\mathcal{Q}$ such that no $X \in \mathcal{F}$ contains $a_0$.

Define $\mathcal{F}^a$ as the output of Algorithm~1 for $\mathcal{Q}_{a_0a, B}$ with pivot $a$.
Since $\mathcal{Q}_{a_0a, B}$ is displayed by $\{ X \cap B \mid a \in X \in \mathcal{F} \} \sqcup \{a\}$,
it holds that $\mathcal{F}^a = \{ X \cap B \mid a \in X \in \mathcal{F} \} \sqcup \{a\}$ by \cref{prop:r=2 unique,prop:algo 1}.
Define $\mathcal{F} \cap B := \{ X \cap B \mid X \in \mathcal{F} \}$.
It can be easily seen that $\mathcal{F} \cap B = \bigcup_{a \in A \setminus \{a_0\}} \{ X \cap B \mid X \in \mathcal{F}^a \}$.
In the following,
we consider to combine $\mathcal{F}^a$s appropriately.

Take any $D \in \mathcal{F} \cap B$,
and define $A_D := \{ a \in A \setminus \{a_0\} \mid \{a\} \cup D \in \mathcal{F}^a \}$.
By the laminarity of $\mathcal{F}$,
$A_D \cup D$ is the unique maximal set $X$ in $\mathcal{F}$ such that $X \cap B = D$.
Hence we can construct the set $\mathcal{G} := \{ A_D \cup D \mid D \in \mathcal{F} \cap B \} \subseteq \mathcal{F}$ from $\mathcal{F}^a$ ($a \in A \setminus \{a_0\}$).
Note that $\mathcal{G}$ is laminar.

All the left is to determine all nonmaximal sets $X \in \mathcal{F}$ with $X \cap B = D$ for each $D \in \mathcal{F} \cap B$.
Fix an arbitrary $D \in \mathcal{F} \cap B$.
Observe that, by the laminarity of $\mathcal{F}$, the set $\{ X \in \mathcal{F} \mid X \cap B = D \}$ is a chain $\{ X_1, X_2, \dots, X_m \}$ with $X_1 \subsetneq X_2 \subsetneq \cdots \subsetneq X_m = A_D \cup D$.
We are going to identify this chain with the help of Algorithm 1. 
Let $X^- := \bigcup\{ X' \in \mathcal{G} \mid X' \subsetneq A_D \cup D \}$,
and choose an arbitrary $b_0 \in B \setminus D$ and $b \in D$.
Note that $X_1 \supseteq X^-$ by the laminarity of $\mathcal{F}$.

Consider that we apply Algorithm~1 to $\mathcal{Q}_{A_D \setminus X^-, b_0b}$ and obtain a chain $\tilde{\mathcal{H}}$.
If $X_1 \cap A \supsetneq X^- \cap A$,
then $\tilde{\mathcal{H}}$ forms of $\{ (X_1 \setminus X^-) \cap A, (X_2 \setminus X^-) \cap A, \dots, (X_m \setminus X^-) \cap A \}  \sqcup \{ b_0, b \}$.
If $X_1 \cap A = X^- \cap A$,
then $\tilde{\mathcal{H}}$ forms of $\{ (X_2 \setminus X^-) \cap A, \dots, (X_m \setminus X^-) \cap A \}  \sqcup \{ b_0, b \}$.
Therefore we need to detect whether the minimal member in $\tilde{\mathcal{H}}$
is $((X_1 \setminus X^-) \cap A) \cup \{ b_0, b \}$ or $((X_2 \setminus X^-) \cap A) \cup \{ b_0, b \}$,
which can be done by constructing $X_1$ individually as follows.

%
Pick any $a \in X^- \cap A$ and retake $b$ from $D \setminus X'$
for maximal $X' \in {\cal G}$ with 
$a \in X' \subseteq X^-$. 
For $a' \in (X_m \setminus X^-) \cap A$,
it cannot happen that $ab_0 || a'b \in \mathcal{Q}$
since all $X \in \mathcal{F}$ containing $a', b$ also contain $a$.
Furthermore we can say that $ab || a'b_0 \in \mathcal{Q}$ if and only if $a' \not\in X_1 (\ni a, b)$.
This implies that $aa' | bb_0 \in \mathcal{Q}$ 
if and only if $a'$ belongs to $X_1$.
That is, it holds that
\begin{align*}
	X_1 \cap A = (X^- \cap A) \cup \{ a' \in A_D \setminus X^- \mid aa' | b_0b \in \mathcal{Q} \}.
\end{align*}

Thus,
if there is $a' \in A_D \setminus X^-$ with $aa' | b_0b \in \mathcal{Q}$,
then it holds $X_1 \cap A \supsetneq X^- \cap A$.
Since $\tilde{\mathcal{H}} = \{ (X_1 \setminus X^-) \cap A, (X_2 \setminus X^-) \cap A, \dots, (X_m \setminus X^-) \cap A \}  \sqcup \{ b_0, b \}$,
we can construct $\{ X_1, X_2, \dots, X_m \}$ from $\tilde{\mathcal{H}}$.
If there is no $a' \in A_D \setminus X^-$ with $aa' | b_0b \in \mathcal{Q}$,
then it holds $X_1 \cap A = X^- \cap A$.
Since $\tilde{\mathcal{H}} = \{ (X_2 \setminus X^-) \cap A, \dots, (X_m \setminus X^-) \cap A \}  \sqcup \{ b_0, b \}$,
we can construct $\{ X_1, X_2, \dots, X_m \}$ from $\tilde{\mathcal{H}}$ and $X_1 = X^- \cup D$.

The formal description of Algorithm~2 is the following;
note that,
if $\mathcal{F}$ is laminar,
then $|\mathcal{F}|$ is at most $2n$ (see e.g., \cite[Theorem 3.5]{book/Schrijver03}).
\begin{description}
	\item[Algorithm~2 (for complete bipartite quartet system):]
	\item[Input:]
	A complete bipartite quartet system $\mathcal{Q}$.
	\item[Output:] Either detect the incompatibility of $\mathcal{Q}$,
	or obtain a laminar $\mathcal{A}$-cut family $\mathcal{F}$ displaying $\mathcal{Q}$.
	\item[Step 1:] Fix an arbitrary $a_0 \in A$.
	For each $a \in A \setminus \{a_0\}$,
	we execute Algorithm~1 for $\mathcal{Q}_{a_0a, B}$ with pivot $a$.
	If Algorithm~1 outputs ``$\mathcal{Q}_{a_0a, B}$ is not compatible'' for some $a$,
	then output ``$\mathcal{Q}$ is not compatible'' and stop.
	Otherwise, obtain the output $\mathcal{F}^a$ for each $a$.
	\item[Step 2:]
	Let $\mathcal{G} := \emptyset$.
	For each $a \in A \setminus \{a_0\}$,
	update $\mathcal{G}$ as
	\begin{align*}
	\mathcal{G} \leftarrow\ &\mathcal{G} \setminus \{ Y \in \mathcal{G} \mid \exists X \in \mathcal{F}^a \textrm{ such that } X \cap B = Y \cap B \}\\
	&\cup \left(\{Y \in \mathcal{G} \mid \exists X \in \mathcal{F}^a \textrm{ such that } X \cap B = Y \cap B \} \sqcup \{a\}\right)\\
	&\cup
	\{ X \in \mathcal{F}^a \mid \nexists Y \in \mathcal{G} \textrm{ such that } X \cap B = Y \cap B \}.
	\end{align*}
	If $|\mathcal{G}| > 2n$ for some $a$,
	then output ``$\mathcal{Q}$ is not compatible'' and stop.
	\item[Step 3:]
	If $\mathcal{G}$ is not laminar,
	then output ``$\mathcal{Q}$ is not compatible'' and stop.
	Otherwise, define $\mathcal{F} := \mathcal{G}$.
	For each $X \in \mathcal{G}$,
	do the following:
	\begin{description}
		\item[3-1:]
		Let $X^- := \bigcup\{ X' \in \mathcal{G} \mid X' \subsetneq X \}$,
		and choose an arbitrary $b_0 \in B \setminus X$ and $b \in X \cap B$.
		\item[3-2:]
		Execute Algorithm~1 for $\mathcal{Q}_{(X \setminus X^-) \cap A, b_0b}$ with pivot $b$.
		If Algorithm~1 outputs ``$\mathcal{Q}_{(X \setminus X^-) \cap A, b_0 b}$ is not compatible,''
		then output ``$\mathcal{Q}$ is not compatible'' and stop.
		Otherwise, define
		\begin{align*}
		\mathcal{H} := \textrm{the output family of Algorithm~1} \sqcup (X^- \cup (X \cap B)).
		\end{align*}
		If $X^- \neq \emptyset$, then go to Step~3-3.
		Otherwise, go to Step~3-4
		\item[3-3:]
		Choose an arbitrary $a \in X^- \cap A$ and retake $b$ from $(X \setminus X') \cap B$
		for maximal $X' \in {\cal G}$ with 
		$a \in X' \subseteq X^-$.
		If there is no $a' \in (X \setminus X^-) \cap A$ with $aa' | b_0b \in \mathcal{Q}$,
		then update $\mathcal{H} \leftarrow \mathcal{H} \cup \{X^- \cup (X \cap B)\}$.
		\item[3-4:]
		$\mathcal{F} \leftarrow \mathcal{F} \cup \mathcal{H}$.
	\end{description}
	\item[Step 4:]
	If $\mathcal{F}$ displays $\mathcal{Q}$,
	then output $\mathcal{F}$.
	Otherwise, output ``$\mathcal{Q}$ is not compatible.''
	\qqed
\end{description}

\begin{prop}\label{prop:algo 2}
	Algorithm~{\rm 2} solves \QCP\ for a complete bipartite quartet system $\mathcal{Q}$ in $O(|\mathcal{Q}|)$ time.
\end{prop}
	\begin{proof}
		(Validity).
		It suffices to show that
		if $\mathcal{Q}$ is compatible then Algorithm~2 outputs an laminar $\mathcal{A}$-cut family $\mathcal{F}$ displaying $\mathcal{Q}$,
		and it follows from the argument before the formal description of Algorithm~2.

		(Complexity).
		We can easily see that, in Steps~1, 3-2, and~3-3, Algorithm~2 checks a quartet tree for $a, a', b, b'$ at most one time for each distinct $a, a' \in A$ and $b, b' \in B$.
		Hence Steps~1, 3-2, and~3-3 can be done in $O(|\mathcal{Q}|)$ time.
		In Step~2, the update takes $O(|\mathcal{F}^{a}| + |\mathcal{G}|) = O(n)$ time for each $a \in A \setminus \{a_0\}$ since $\mathcal{F}^a$ is a chain.
		Hence Step~2 can be done in $O(n^2)$ time.
		Checking the laminarity of $\mathcal{G}$ can be done in $O(n^2)$ time since $|\mathcal{G}| = O(n)$.
		For each $X \in \mathcal{G}$,
		it takes $O(n)$ time to construct $X^-$;
		Step~3-1 can be done in $O(n^2)$ time.
		Step~4 can be done in $O(|\mathcal{Q}|)$ time.
		Thus
		the running-time of Algorithm~1 is $O(|\mathcal{Q}| + n^2) = O(|\mathcal{Q}|)$
		since $|\mathcal{Q}| = \binom{|A|}{2}\binom{|B|}{2} = \Omega(n^2)$.
	\end{proof}

\subsection{Algorithm for complete multipartite quartet system}\label{subsec:complete algo}
In this subsection, we present a polynomial-time algorithm for complete multipartite quartet systems.
Let $\mathcal{A} := \{A_1, A_2, \dots, A_r\}$ be a partition of $[n]$ with $|A_i| \geq 2$ for all $i \in [r]$.
For the analysis of the running-time of Algorithm~4,
we assume $|A_1| \geq |A_2| \geq \cdots \geq |A_r|$.
For $R \subseteq [r]$ with $|R| \geq 2$, let $\mathcal{A}_R := \{A_i\}_{i \in R}$ and $A_R := \bigcup_{i \in R} A_i$.
For a complete $\mathcal{A}$-partite quartet system $\mathcal{Q} = \bigcup_{1 \leq i < j \leq r} \mathcal{Q}_{ij}$,
define $\mathcal{Q}_R := \bigcup_{i,j \in R,i<j} \mathcal{Q}_{ij}$.
That is, $\mathcal{Q}_R$ is the complete $\mathcal{A}_R$-partite quartet system included in $\mathcal{Q}$.
For an $\mathcal{A}$-cut family $\mathcal{F}$,
define $\mathcal{F}_R := \{ X \cap A_R \mid X \in \mathcal{F} \textrm{ such that } X \cap A_R \textrm{ is an $\mathcal{A}_R$-cut} \}$.
Note that $\mathcal{F}_R$ is an $\mathcal{A}_R$-cut family.
Then we can easily see the following lemma,
which says that partial information $\mathcal{F}_R$ of $\mathcal{F}$ can be obtained from $\mathcal{Q}_R$.
\begin{lem}\label{lem:restriction}
	Suppose $R \subseteq [r]$ with $|R| \geq 2$.
	If $\mathcal{Q}$ is displayed by $\mathcal{F}$,
	then $\mathcal{Q}_R$ is displayed by $\mathcal{F}_R$.
	Furthermore, if $\mathcal{Q}$ is compatible,
	then so is $\mathcal{Q}_R$.
\end{lem}

Our algorithm for {\sc Displaying} is to construct an $\mathcal{A}_{[t]}$-cut family $\mathcal{F}^{(t)}$ displaying $\mathcal{Q}_{[t]}$
for $t = 2, 3, \dots, r$ in turn as follows.
\begin{itemize}
	\item First we obtain an $\mathcal{A}_{12}$-cut family $\mathcal{F}^{(2)}$ displaying $\mathcal{Q}_{12}$ by Algorithm~2.
	\item For $t \geq 3$, we extend an $\mathcal{A}_{[t-1]}$-cut family $\mathcal{F}^{(t-1)}$ displaying $\mathcal{Q}_{[t-1]}$ to an $\mathcal{A}_{[t]}$-cut family $\mathcal{F}^{(t)}$ displaying $\mathcal{Q}_{[t]}$ by Algorithm~3.
	In order to construct $\mathcal{F}^{(t)}$ in Algorithm~3,
	we use an $\mathcal{A}_{it}$-cut family $\mathcal{G}^{(i)}$ displaying $\mathcal{Q}_{it}$ for all $i \in [t-1]$.
	These $\mathcal{G}^{(i)}$ are obtained by Algorithm~2.
	\item We perform the above extension step from $t=3$ to $r$,
	and then obtain a desired $\mathcal{A}$-cut family $\mathcal{F} := \mathcal{F}^{(r)}$.
	This is described in Algorithm~4.
\end{itemize}

For nonempty $R \subseteq [r]$,
we define $\sim_R$ for $\mathcal{A}$-cuts by:
\begin{align*}
	X \sim_R Y \iff \{ \ave{X}_R \cap X, \ave{X}_R \setminus X \} = \{ \ave{Y}_R \cap Y, \ave{Y}_R \setminus Y \},
\end{align*}
where $\ave{X}_R := \ave{X} \cap A_R$ and $\ave{Y}_R := \ave{Y} \cap A_R$.
Note that $R$ can be a singleton.
We define a partial order relation $\prec$ in $\mathcal{A}$-cuts by:
$X \prec Y$ if $\ave{X} \subsetneq \ave{Y}$ and $\{ \ave{X} \cap X, \ave{X} \setminus X \} = \{ \ave{X} \cap Y, \ave{X} \setminus Y \}$.
Define $X \preceq Y$ by $X \prec Y$ or $X \sim Y$.
Note that, if $\{X\}$ displays $ab || a'b'$ and $X \preceq Y$,
then $\{Y\}$ also displays $ab || a'b'$.

We describe the extension step for $t = r$, i.e., from $\mathcal{F}^{(r-1)}$ to $\mathcal{F}^{(r)}$.
Algorithm~3 constructs a minimal laminarizable family $\mathcal{F}^{(t)}$ displaying $\mathcal{Q}_{[t]}$ from a minimal laminarizable family $\mathcal{F}^{(t-1)}$ displaying $\mathcal{Q}_{[t-1]}$.
It is noted that,
if $\mathcal{F}$ is laminarizable and $X \not\sim Y$ for all distinct $X, Y \in \mathcal{F}$,
then $|\mathcal{F}|$ is at most $2n = 2|A_{[r]}|$.
\begin{description}
	\item[Algorithm 3 (for extending $\mathcal{F}'$ to $\mathcal{F}$):]
	\item[Input:] A complete $\mathcal{A}$-partite quartet system $\mathcal{Q}$ and an $\mathcal{A}_{[r-1]}$-cut family $\mathcal{F}'$ displaying $\mathcal{Q}_{[r-1]}$,
	where $|\mathcal{F}'| \leq 2|A_{[r-1]}|$.
	\item[Output:] 
	Either detect the incompatibility of $\mathcal{Q}$,
	or obtain an $\mathcal{A}$-cut family $\mathcal{F}$ displaying $\mathcal{Q}$, where $|\mathcal{F}| \leq 2n = 2|A_{[r]}|$.
	\item[Step 1:] 
	For each $i \in [r-1]$,
	execute Algorithm~2 for $\mathcal{Q}_{ir}$.
	If Algorithm~2 returns ``$\mathcal{Q}_{ir}$ is not compatible'' for some $i \in [r-1]$,
	then output ``$\mathcal{Q}$ is not compatible''
	and stop.
	Otherwise,
	obtain $\mathcal{G}^{(i)}$ for all $i \in [r-1]$.
	Let $\mathcal{F} := \emptyset$.
	\item[Step 2:] If $\mathcal{F}' = \emptyset$, update as $\mathcal{F} \leftarrow \mathcal{F} \cup \bigcup_{i \in [r-1]} \mathcal{G}^{(i)}$,
	and go to Step~3.
	Otherwise, do the following:
	Take any $X' \in \mathcal{F}'$. 
	Suppose $\ave{X'} = A_S$ for some $S \subseteq [r-1]$.
	Let $\mathcal{F}^{X'}$ be the set of maximal $\mathcal{A}$-cuts $Y$ with respect to $\prec$
	satisfying the following:
	\begin{itemize}
		\item There are $R \subseteq S$ and $X_i \in \mathcal{G}^{(i)}$ for $i \in R$ such that $\ave{Y} = A_{R \cup \{r\}}$, $Y \sim_{R} X'$, and $Y \sim_{i r} X_i$ for all $i \in R$.
	\end{itemize}
	
	Then update as
	$\mathcal{F} \leftarrow \mathcal{F} \cup \{X'\} \cup \mathcal{F}^{X'}$ and $\mathcal{F}' \leftarrow \mathcal{F}' \setminus \{ X'\}$,
	and go to Step~2.
	\item[Step 3:]
	Update as
	\begin{align*}
	\mathcal{F} \leftarrow \textrm{the set of maximal elements in $\mathcal{F}$ with respect to $\prec$}.
	\end{align*}
	If $|\mathcal{F}| \leq 2n$,
	then output $\mathcal{F}$.
	Otherwise, output ``$\mathcal{Q}$ is not compatible.''
	\qqed
\end{description}

We show the correctness of Algorithm~3.
We first introduce the following lemma that gives a sufficient condition for the non-laminarizability of an $\mathcal{A}$-cut family.
For simplicity, we denote $X \sim_{\{i\}} Y$, $X \sim_{\{i,j\}} Y$, and $X \sim_{\{i,j,k\}} Y$ by $X \sim_{i} Y$, $X \sim_{ij} Y$, and $X \sim_{ijk} Y$,
respectively.
\begin{lem}\label{lem:nonlaminar}
	Suppose that $X, Y, Z$ are $\mathcal{A}$-cuts with $\ave{X} \supseteq A_{ij}$, $\ave{Y} \supseteq A_{ik}$, and $\ave{Z} \supseteq A_{jk}$ for some distinct $i,j,k \in [r]$.
	If there is an $\mathcal{A}$-cut $W$ such that $X \sim_{ij} W$, $Y \sim_{ik} W$,
	$Z \sim_{jk} W$, and $X,Y,Z \not\sim_{ijk} W$,
	then $\{X, Y, Z\}$ is not laminarizable.
\end{lem}
\begin{proof}
	We show the contrapositive;
	suppose that $\{X, Y, Z\}$ is laminarizable,
	where $X \sim_{ij} W$, $Y \sim_{ik} W$, $Z \sim_{jk} W$, and $Y, Z \not\sim_{ijk} W$ hold for some $\mathcal{A}$-cut $W$ and some distinct $i,j,k \in [r]$.
	We prove $X \sim_{ijk} W$.
	
	By the laminarizability of $\{X, Y, Z\}$,
	there is a laminar $\mathcal{A}$-cut family equivalent to $\{X, Y, Z\}$.
	We also denote it by $\{X, Y, Z\}$,
	that is, $\{X, Y, Z\}$ is laminar.
	Moreover, by $X \sim_{ij} W$, $Y \sim_{ik} W$, and $Z \sim_{jk} W$,
	we can assume $X \cap A_{ij} = W \cap A_{ij}$, $Y \cap A_{ik} = W \cap A_{ik}$, and $Z \cap A_{jk} = W \cap A_{jk}$.
	These imply $X \cap A_i = Y \cap A_i (\neq \emptyset)$, $X \cap A_j = Z \cap A_j (\neq \emptyset)$, and $Y \cap A_k = Z \cap A_k (\neq \emptyset)$.
	By $Y \cap Z \neq \emptyset$ and the laminarity of $\{Y, Z\}$,
	either $Y \subsetneq Z$ or $Y \supsetneq Z$ holds.
	Without loss of generality,
	we assume $Y \subsetneq Z$.
	
	By $Z \not\sim_{ijk} W$ and $Y \subsetneq Z$,
	it must hold $Z \cap A_i \supsetneq Y \cap A_i = X \cap A_i$.
	Thus we have $Z \cap A_{ij} \supsetneq X \cap A_{ij}$.
	By $Y \not\sim_{ijk} W$ and $Y \subsetneq Z$,
	it must hold $Y \cap A_j \subsetneq Z \cap A_j = X \cap A_j$.
	Thus we have $Y \cap A_{ij} \subsetneq X \cap A_{ij}$.
	Then, by the laminarity of $\{X, Y, Z\}$,
	we obtain $Z \supsetneq X \supsetneq Y$.
	Since $Y \cap A_k = Z \cap A_k$ holds,
	we have $X \cap A_k = Y \cap A_k = Z \cap A_k$.
	This implies $X \sim_{ijk} W$.
\end{proof}

As a compatible complete bipartite quartet system (\cref{prop:r=2 unique}),
a compatible complete multipartite quartet system $\mathcal{Q}$ induces the following uniqueness of a laminarizable family displaying $\mathcal{Q}$,
which ensures the validity of our proposed algorithm.
\begin{prop}\label{prop:r>=3 unique}
	Suppose that a complete $\mathcal{A}$-partite quartet system $\mathcal{Q}$ is compatible.
	Then a minimal laminarizable $\mathcal{A}$-cut family $\mathcal{F}$ displaying $\mathcal{Q}$ is uniquely determined up to $\sim$.
\end{prop}
\begin{proof}
	Suppose that $\mathcal{F}$ and $\mathcal{F}'$ are laminarizable $\mathcal{A}$-cut families displaying $\mathcal{Q}$.
	Let $\mathcal{G} := \{X \in \mathcal{F} \mid \exists X' \in \mathcal{F}' \textrm{ satisfying } X \sim X' \}$.
	It suffices to show that $\mathcal{G}$ is also a laminarizable family displaying $\mathcal{Q}$.
	The laminarizability of $\mathcal{G}$ is clear by $\mathcal{G} \subseteq \mathcal{F}$.
	Hence we show that $\mathcal{G}$ displays $\mathcal{Q}$.
	
	Take any $ab || a' b' \in \mathcal{Q}$.
	Assume that $a, a' \in A_1$ and $b, b' \in A_2$.
	By \cref{lem:restriction},
	for all distinct $i, j \in [r]$,
	$\mathcal{F}_{ij}$ and $\mathcal{F}'_{ij}$ display $\mathcal{Q}_{ij}$.
	By \cref{prop:r=2 unique} and the laminarizability of $\mathcal{F}_{ij}$ and $\mathcal{F}'_{ij}$,
	it holds that $\mathcal{F}_{ij} \sim_{ij} \mathcal{F}'_{ij}$, implying $\mathcal{F}_{ij} \sim \mathcal{F}'_{ij}$.
	In particular,
	$\mathcal{F}_{12}$ and $\mathcal{F}'_{12}$ display $\mathcal{Q}_{12}$, and $\mathcal{F}_{12} \sim \mathcal{F}'_{12}$.
	By $ab || a'b' \in \mathcal{Q}$,
	there are $W \in \mathcal{F}_{12}$ and $W' \in \mathcal{F}'_{12}$ with $W \sim W'$
	such that $\{W\}$ and $\{W'\}$ display $ab || a'b'$.
	We denote by $\mathcal{F}(W)$ (resp. $\mathcal{F}'(W')$) the set of $X \in \mathcal{F}$ (resp. $X' \in \mathcal{F}'$) such that $X \sim_{12} W$ (resp. $X' \sim_{12} W'$).
	Note that $\mathcal{F}(W)$ and $\mathcal{F}'(W')$ are nonempty.
	It suffices to show that there are $X \in \mathcal{F}(W)$ and $X' \in \mathcal{F}'(W')$ satisfying $X \sim X'$.
	Then $\mathcal{G}$ contains $X$,
	implying that $\mathcal{G}$ displays $ab || a'b'$.

	Suppose, to the contrary, that there is no such a pair.
	Let $R$ be a maximal subset of $[r]$ such that $X \sim_R X'$ for some $X \in \mathcal{F}(W)$ and $X' \in \mathcal{F}'(W')$,
	and take such $X$ and $X'$.
	Note that $|R| \geq 2$ since $R \supseteq \{1,2\}$.
	Since $X \not\sim X'$,
	there is $k \in [r] \setminus R$ such that $\ave{X'} \supseteq A_{k}$ and $X \not\sim_{k} X'$.
	Furthermore, by $X' \cap A_{i k} \in \mathcal{F}'_{i k}$ and $\mathcal{F}_{i k} \sim \mathcal{F}'_{i k}$,
	for each $i \in R$
	there is $Y \in \mathcal{F}$ with $Y \sim_{i k} X'$.
	
	Take $Y \in \mathcal{F}$ so that $I := \{i \in R \mid Y \sim_{i k} X' \}$ is maximal.
	By the maximality of $R$,
	we have $R \setminus I \neq \emptyset$;
	otherwise $X' \sim_{R \cup \{k\}} Y$ and $Y  \in \mathcal{F}(W)$, contradicting the maximality of $R$.
	Choose arbitrary $j \in R \setminus I$.
	Then there is $Z \in \mathcal{F}$ with $Z \sim_{j k} X'$.
	Furthermore, by the maximality of $I$,
	there is $i \in I$ such that $Z \not\sim_{i k} X'$.
	Thus we have $X \sim_{ij} X'$, $Y \sim_{i k} X'$, $Z \sim_{j k} X'$, and $X, Y, Z \not\sim_{ijk} X'$.
	Hence, by \cref{lem:nonlaminar},
	$\{X, Y, Z\} \subseteq \mathcal{F}$ is not laminarizable,
	contradicting the laminarizability of $\mathcal{F}$.
\end{proof}

We are now ready to prove the validity of Algorithm~3.
\begin{prop}\label{prop:algo 3}
\begin{description}
    \item[{\rm (1)}]
    If Algorithm~{\rm 3} outputs $\mathcal{F}$,
	then $\mathcal{F}$ displays $\mathcal{Q}$.
	In addition,
	if $\mathcal{Q}$ is compatible and $\mathcal{F}'$ is a minimal laminarizable $\mathcal{A}_{[r-1]}$-cut family displaying $\mathcal{Q}_{[r-1]}$,
	then $\mathcal{F}$ is a minimal laminarizable $\mathcal{A}$-cut family.
    \item[{\rm (2)}]
    The running-time of Algorithm~{\rm 3} is $O(n^4)$.
\end{description}
\end{prop}
\begin{proof}
	(1).
	Suppose that Algorithm~3 outputs $\mathcal{F}$.
	It can be easily seen that
	$\mathcal{F}_{ij} \sim \mathcal{F}'_{ij}$ for all distinct $i, j \in [r-1]$,
	and $\mathcal{F}_{ir} \sim \mathcal{G}^{(i)}$ for all $i \in [r-1]$.
	This implies that $\mathcal{F}$ displays $\mathcal{Q}_{[r-1]} \cup \bigcup_{i \in [r-1]} \mathcal{Q}_{ir} = \mathcal{Q}$.
	
	Suppose that $\mathcal{Q}$ is compatible, and that $\mathcal{F}'$ is a minimal laminarizable $\mathcal{A}$-cut family displaying $\mathcal{Q}_{[r-1]}$.
	Let $\mathcal{F}^*$ be a minimal laminarizable $\mathcal{A}$-cut family displaying $\mathcal{Q}$.
	It suffices to show that $\mathcal{F} / {\sim} \subseteq \mathcal{F}^* / {\sim}$.
	Indeed, such $\mathcal{F}$ is laminarizable,
	and the minimality of $\mathcal{F}^*$ implies $\mathcal{F} \sim \mathcal{F}^*$.
	Note that, by \cref{prop:r=2 unique} and \cref{lem:restriction},
	it holds that $\mathcal{F}^*_{ij} \sim_{ij} \mathcal{F}'_{ij}$ for all distinct $i, j \in [r]$ and $\mathcal{F}^*_{ir} \sim_{ir} \mathcal{G}^{(i)}$ for all $i \in [r-1]$.

	Take any $X \in \mathcal{F}$.
	There are two cases (i) $\ave{X} \subseteq A_{[r-1]}$ and (ii) $\ave{X} \supseteq A_r$.
	We show in the both cases that there is $X^* \in \mathcal{F}^*$ satisfying $X \sim X^*$.
	By \cref{lem:restriction},
	$\mathcal{F}^*_{[r-1]}$ is a laminarizable family displaying $\mathcal{Q}_{[r-1]}$.
	Hence, by \cref{prop:r>=3 unique} and the minimality of $\mathcal{F}'$,
	we have $\mathcal{F}' / {\sim} \subseteq \mathcal{F}^*_{[r-1]} / {\sim}$.
	
	(i). $\ave{X} \subseteq A_{[r-1]}$ implies $X \in \mathcal{F}'$.
	Hence, by $\mathcal{F}' / {\sim} \subseteq \mathcal{F}^*_{[r-1]} / {\sim}$, there is $X^* \in \mathcal{F}^*$ such that $X \sim_{[r-1]} X^*$.
	Suppose, to the contrary, that such $X^*$ satisfies $\ave{X^*} \supseteq A_r$, i.e., $X \prec X^*$.
	Then there must be $X^* \cap A_{ir} \in \mathcal{F}^*_{ir}$ for every $i \in [r-1]$ with $\ave{X^*} \supseteq A_i$.
	By $\mathcal{F}^*_{ir} \sim \mathcal{G}^{(i)}$ for all $i \in [r-1]$, $\mathcal{F}^X$ contains $Y$ satisfying $Y \succeq X^* (\succ X)$ in Step~2.
	Hence $X$ is deleted from $\mathcal{F}$ in Step~3, contradicting $X \in \mathcal{F}$.
	Thus $X \sim X^* \in \mathcal{F}^*$ holds.
	
	(ii).
	Suppose, to the contrary, that there is no $X^* \in \mathcal{F}^*$ satisfying $X \sim X^*$.
	Let $X'$ be an element in $\mathcal{F}'$ that is used to construct $X$ in Step~2.
	By $\mathcal{F}' / {\sim} \subseteq \mathcal{F}^*_{[r-1]} / {\sim}$,
	there is $X^* \in \mathcal{F}^*$ with $X^* \sim_{[r-1]} X'$.
	By the maximality of $X$ and $X^* \not\sim X$,
	we have $X^* \not\sim_r X$.
	Let $R \subseteq [r-1]$ be the set of indices with $\ave{X} = A_{R \cup \{r\}}$.
	Take $Y^* \in \mathcal{F}^*$ so that the set $I \subseteq R$ with $Y^* \sim_{I \cup \{r\}} X$ is maximal.
	Note that $I \neq \emptyset$, since it holds that $X \cap A_{i r} \in \mathcal{G}^{(i)}$ and $\mathcal{G}^{(i)} \sim \mathcal{F}^*_{ir}$ for $i \in R$.
	
	Assume that $I \neq R$.
	Choose arbitrary $j \in R \setminus I$.
	Then there is $Z^* \in \mathcal{F}^*$ with $Z^* \sim_{j r} X$.
	Furthermore, by the maximality of $I$,
	there is $i \in I$ such that $Z^* \not\sim_{i r} X$.
	Hence we have $X^* \sim_{ij} X' \sim_{ij} X$, $Y^* \sim_{ir} X$, $Z^* \sim_{j r} X$, and $X^*, Y^*, Z^* \not\sim_{ijr} X$.
	Thus, by \cref{lem:nonlaminar},
	$\{X^*, Y^*, Z^*\} \subseteq \mathcal{F}^*$ is not laminarizable,
	contradicting the laminarizability of $\mathcal{F}^*$.
	
	Assume that $I = R$.
	Since $Y^* \not\sim X$,
	there is $k \in [r-1] \setminus R$ with $\ave{Y^*} \supseteq A_{k}$.
	For such $k$, we have $X' \not\sim_{k} Y^*$;
	otherwise, by $Y^* \cap A_{kr} \in \mathcal{F}^*_{kr}$ and $\mathcal{F}^*_{k r} \sim \mathcal{G}^{(k)}$,
	$Y$ satisfying $Y \succeq Y^* \cap A_{R \cup \{k, r\}} (\succ X)$ is constructed in Step~2 for $X'$,
	contradicting the maximality of $X$.
	Take $Y' \in \mathcal{F}'$ so that the set $I' \subseteq R$ with $Y' \sim_{I' \cup \{k\}} Y^*$ is maximal.
	Note that $I' \neq \emptyset$, since it holds that $Y^* \cap A_{i k} \in \mathcal{F}^*_{i k}$ and $\mathcal{F}^*_{i k} \sim \mathcal{F}'_{i k}$ for $i \in R$.
	Furthermore $I' \neq R$, since otherwise,
	by $Y^* \cap A_{kr} \in \mathcal{F}^*_{kr}$ and $\mathcal{F}^*_{kr} \sim \mathcal{G}^{(k)}$,
	$Y$ satisfying $Y \succeq Y^* \cap A_{R \cup \{k, r\}} (\succ X)$ is constructed in Step~2 for $Y'$,
	contradicting the maximality of $X$.
	
	Choose arbitrary $j \in R \setminus I'$.
	Then there is $Z' \in \mathcal{F}'$ with $Z' \sim_{j k} Y^*$.
	Furthermore, by the maximality of $I'$,
	there is $i \in I'$ such that $Z' \not\sim_{i k} Y^*$.
	Hence we have $X' \sim_{ij} Y^*$, $Y' \sim_{ik} Y^*$, $Z' \sim_{j k} Y^*$, and $X', Y', Z' \not\sim_{ijk} Y^*$.
	Thus, by \cref{lem:nonlaminar},
	$\{X', Y', Z'\} \subseteq \mathcal{F}'$ is not laminarizable,
	contradicting the laminarizability of $\mathcal{F}'$.
	
	(2).
	Step~1 can be done in $O(\sum_{i \in [r-1]}|\mathcal{Q}_{ir}|)$ time by \cref{prop:algo 2}.
	Step~2 can be done in $O(rn^2)$ time by using the structure of $\mathcal{F}_{ir}$ as follows.
	
	First we consider the running-time of one iteration in Step~2 for $X' \in \mathcal{F}'$.
	By the assumption $|A_1| \geq |A_2| \geq \cdots \geq |A_r|$,
	it holds $r|A_r| = O(n)$.
	Recall $\ave{X'} = A_S$.
	Without loss of generality,
	we can assume that if $X \in \mathcal{G}^{(p)}$ satisfies $X \sim_{p} X'$ for some $p \in S$,
	then $X \cap A_{p} = X' \cap A_{p}$.
	For each $p \in S$,
	construct $\mathcal{F}_p := \{ X \cap A_r \mid X \in \mathcal{G}^{(p)},\ X \cap A_{p} = X' \cap A_{p} \}$ in $O(|\bigcup_{p \in S} \mathcal{G}^{(p)}|) = O(\sum_{i \in [r-1]}|A_{ir}|) = O(n)$ time.
	By the laminarity of $\mathcal{G}^{(p)}$,
	$\mathcal{F}_p$ is a chain $\{ F_p^1, F_p^2, \dots, F_p^{k} \}$ for $p \in S$,
	where $F_p^1 \supsetneq F_p^2 \supsetneq \cdots \supsetneq F_p^{k}$
	(this chain can be obtained while constructing $\mathcal{G}^{(p)}$ in Algorithm~1).
	Take any maximal element $F$ in $\bigcup_{p \in S} \mathcal{F}_p$,
	and obtain the set $R := \{p \mid F \in \mathcal{F}_p \}$.
	This can be done in $O(r)$ time.
	Here, if $|R| \geq 2$,
	construct $Y$ with $\ave{Y} = A_{R \cup \{r\}}$, $Y \cap A_R = X' \cap A_R$, and $Y \cap A_r = F$,
	and add $Y$ to $\mathcal{F}^{X'}$.
	Then, for each $p$ with $F = F_p^1$,
	update $\mathcal{F}_p \leftarrow \mathcal{F}_p \setminus \{F_p^1\}$,
	and do the same thing.
	By repeating this procedure in $O(|\bigcup_{p \in S} \mathcal{F}_p|) = O(n)$ times,
	we can construct the set $\mathcal{F}^{X'}$.
	Hence one iteration in Step~2 takes $O(rn)$ time.
	Since $\mathcal{F}' = O(|A_{[r-1]}|) = O(n)$,
	the number of iterations in Step~2 is $O(n)$.
	Thus the running-time of Step~2 is bounded by $O(rn^2)$.
	
	Step~3 can be done in $O(n^4)$ time.
	Indeed, by the construction of $\mathcal{F}^X$ in Step~2 as above,
	$|\mathcal{F}^X| = O(n)$.
	Hence, before updating $\mathcal{F}$,
	it holds that
	$|\mathcal{F}| = O(n) + O(n^2) + O(n) = O(n^2)$,
	where the first, second, and last terms come from $\mathcal{F}'$, $\bigcup_{X' \in \mathcal{F}'}\mathcal{F}^{X'}$, and $\bigcup_{i \in [r-1]} \mathcal{G}^{(i)}$,
	respectively.
	Since it takes $O(|\mathcal{F}|^2)$ time to update $\mathcal{F}$ in Step~3,
	the running-time of Step~3 is bounded by $O(n^4)$.
	
	Thus Algorithm~3 runs in $O(\sum_{i \in [r-1]}|\mathcal{Q}_{ir}| + rn^2 + n^4) = O(n^4)$ time.
	Indeed, $O(\sum_{i \in [r-1]}|\mathcal{Q}_{ir}|) = O(|A_{[r-1]}|^2|A_r|^2) = O(n^4)$
	by $|\mathcal{Q}_{ij}| = \binom{|A_i|}{2}\binom{|A_j|}{2} = O(|A_i|^2 |A_j|^2)$.
\end{proof}
	
Our proposed algorithm for {\sc Displaying} is the following.
\begin{description}
	\item[Algorithm~4 (for {\sc Displaying}):]
	\item[Step 1:]
	Execute Algorithm~2 for $\mathcal{Q}_{12}$.
	If Algorithm~2 returns ``$\mathcal{Q}_{12}$ is not compatible,''
	then output ``$\mathcal{Q}$ is not compatible''
	and stop.
	Otherwise, obtain $\mathcal{F}^{(2)}$.
	\item[Step 2:]
	For $t = 3, \dots, r$,
	execute Algorithm~3 for $\mathcal{F}^{(t-1)}$.
	If Algorithm~3 returns ``$\mathcal{Q}_{[t]}$ is not compatible,''
	then
	output ``$\mathcal{Q}$ is not compatible''
	and stop.
	Otherwise, obtain $\mathcal{F}^{(t)}$.
	\item[Step 3:]
	Output $\mathcal{F} := \mathcal{F}^{(r)}$.
	\qqed
\end{description}

\begin{thm}\label{thm:algo 4}
	Algorithm~{\rm 4} solves {\sc Displaying} in $O(rn^4)$ time.
	Furthermore, if the input is compatible,
	then the output is a minimal laminarizable $\mathcal{A}$-cut family.
\end{thm}
	\begin{proof}
		(Validity).
		Suppose that $\mathcal{Q}$ is compatible.
		Then so is $\mathcal{Q}_{12}$ by \cref{lem:restriction}.
		We can obtain appropriate $\mathcal{F}^{(2)}$ in Step~1 by \cref{prop:algo 2}.
		Furthermore, by \cref{prop:algo 2},
		Step~2 constructs appropriate $\mathcal{F}^{(t)}$ for all $t = 3, \dots, r$.
		Hence, in Step~3, the algorithm outputs a minimal laminarizable family $\mathcal{F}$.
		This implies that, if Algorithm~3 outputs ``$\mathcal{Q}$ is not compatible,''
		then actually $\mathcal{Q}$ is not compatible.
		
		Even if $\mathcal{Q}$ is not compatible and Algorithm~2 outputs some $\mathcal{A}$-cut family $\mathcal{F}$,
		\cref{prop:algo 3} justifies that $\mathcal{F}$ displays $\mathcal{Q}$
		since $\mathcal{F}^{(2)}$ displays $\mathcal{Q}_{12}$ by \cref{prop:algo 2}.
		
		(Complexity).
		Note that $t|A_{t}| = O(|A_{[t]}|)$ for $t = 3, \dots, r$.
		Steps~1 and~2 take $O(|\mathcal{Q}_{12}| + \sum_{t = 3}^r |A_{[t]}|^4) = O(r|A_{[r]}|^4) = O(rn^4)$ time.
		Thus the running-time of Algorithm~3 is $O(rn^4)$.
	\end{proof}

\section{Full multipartite quartet system}\label{sec:full}
\subsection{{\sc Full Displaying} and {\sc Full Laminarization}}\label{subsec:full preliminary}
As in \cref{subsec:complete preliminary},
we see that \QCP\ for full multipartite quartet systems can be divided into two subproblems named as {\sc Full Displaying} and {\sc Full Laminarization}.
The outline of the argument is the same as the case of complete multipartite quartet systems.
Let $\mathcal{Q}$ be a full quartet system on finite set $A \subseteq [n]$.
We say that a family $\mathcal{F} \subseteq 2^{[n]}$ {\it displays} $\mathcal{Q}$
if for all distinct $a, b, c, d \in A$,
the following (i) and (ii) are equivalent:
\begin{itemize}
	\item[(i)] $ab || cd$ belongs to $\mathcal{Q}$.
	\item[(ii)] There is $X \in \mathcal{F}$ satisfying $a, b \in X \not\ni c, d$ or $a, b \not\in X \ni c, d$.
\end{itemize}
Let $\mathcal{A} := \{A_1, A_2, \dots, A_r\}$ be a partition of $[n]$
with $|A_i| \geq 2$ for all $i \in [r]$,
and $\mathcal{Q} = \mathcal{Q}_0 \cup \mathcal{Q}_1 \cup \cdots \cup \mathcal{Q}_r$ be a full $\mathcal{A}$-partite quartet system.
Here $\mathcal{Q}_0$ is complete $\mathcal{A}$-partite and
$\mathcal{Q}_i$ is full on $A_i$ for each $i \in [r]$.
We also say that $\mathcal{F}$ {\it displays} $\mathcal{Q}$
if $\mathcal{F}$ displays all $\mathcal{Q}_0, \mathcal{Q}_1, \dots, \mathcal{Q}_r$.
We can see that a full $\mathcal{A}$-partite system $\mathcal{Q}$ is compatible
if and only if there exists a laminar family $\mathcal{L}$ displaying $\mathcal{Q}$,
and that, in the compatible case,
$\mathcal{Q}$ is displayed by
the phylogenetic tree corresponding to $\mathcal{L}$.
Therefore \QCP\ for full $\mathcal{A}$-partite quartet system $\mathcal{Q}$ can be viewed as the problem of finding a laminar family $\mathcal{L}$ displaying $\mathcal{Q}$ if it exists.


The equivalent relation $\sim$ introduced in \cref{subsec:complete preliminary}
can also capture the equivalence of full $\mathcal{A}$-partite quartet systems displayed by families.
A set $X \subseteq [n]$ is called a {\it weak $\mathcal{A}$-cut}
if $X$ is an $\mathcal{A}$-cut, or $\ave{X} = A_i$ for some $i \in [r]$ and $\min \{ |X|, |A_i \setminus X| \} \geq 2$.
One can see that $X$ is a weak $\mathcal{A}$-cut
if and only if $\{X\}$ and $\{\emptyset\}$ display different full $\mathcal{A}$-partite quartet systems,
and
that for weak $\mathcal{A}$-cuts $X$ and $Y$, $X \sim Y$
if and only if $\{X\}$ and $\{ Y \}$ display the same full $\mathcal{A}$-partite quartet system.
For weak $\mathcal{A}$-cut families $\mathcal{F}$ and $\mathcal{G}$,
if $\mathcal{F} \sim \mathcal{G}$
then both $\mathcal{F}$ and $\mathcal{G}$ display the same full $\mathcal{A}$-partite quartet system.


By the above argument,
\QCP\ for a full $\mathcal{A}$-partite quartet system $\mathcal{Q}$ can be divided into the following two subproblems.
\begin{description}
	\item[\underline{\sc Full Displaying}]
	\item[Given:] A full $\mathcal{A}$-partite quartet system $\mathcal{Q}$.
	\item[Problem:] Either detect the incompatibility of $\mathcal{Q}$,
	or obtain some weak $\mathcal{A}$-cut family $\mathcal{F}$ displaying $\mathcal{Q}$.
	In addition, if $\mathcal{Q}$ is compatible,
	then $\mathcal{F}$ should be laminarizable.
	\item[\underline{\sc Full Laminarization}]
	\item[Given:] A weak $\mathcal{A}$-cut family $\mathcal{F}$.
	\item[Problem:] Determine whether
	$\mathcal{F}$ is laminarizable or not.
	If $\mathcal{F}$ is laminarizable,
	obtain a laminar weak $\mathcal{A}$-cut family $\mathcal{L}$ with 
	$\mathcal{L} \sim \mathcal{F}$.
\end{description}
Here, in {\sc Full Laminarization}, we assume that no distinct $X, Y$ with $X \sim Y$ are contained in $\mathcal{F}$,
i.e., $|\mathcal{F}| = |\mathcal{F} / {\sim}|$.

{\sc Full Laminarization} can be solved in $O(n^4)$ time
by reducing to {\sc Laminarization}.
\begin{thm}\label{thm:full laminarization}
	{\sc Full Laminarization} can be solved in $O(n^4)$ time.
\end{thm}
	\begin{proof}
		Let $\mathcal{F}$ be the input weak $\mathcal{A}$-cut family of {\sc Full Displaying}.
		We can assume $|\mathcal{F}| \leq 2n$
		since otherwise $\mathcal{F}$ is not laminarizable.
		For each $X \in \mathcal{F}$,
		we add a new set $A^X$ with $|A^X| = 2$ to the ground set $[n]$ and to the partition $\mathcal{A}$ of $[n]$;
		the ground set will be $[n] \cup \bigcup_{X \in \mathcal{F}} A^X$ and the partition will be $\mathcal{A}_+ := \mathcal{A} \cup \{ A^X \mid X \in \mathcal{F}\}$.
		Note that the size of the new ground set is $O(n)$
		by $|\mathcal{F}| \leq 2n$ and $|A^X| = 2$.
		Define $X_+ := X \cup \{x\}$, where $x$ is one of the two elements of $A^X$
		and
		define $\mathcal{F}_+ := \{ X_+ \mid X \in \mathcal{F} \}$.
		Since $\ave{X_+} = \ave{X} \cup A^X$,
		$\mathcal{F}_+$ is an $\mathcal{A}_+$-cut family.
		It is easily seen that there exists a laminar family $\mathcal{L}$ with $\mathcal{L} \sim \mathcal{F}$ if and only if there exists a laminar family $\mathcal{L}_+$ with $\mathcal{L}_+ \sim \mathcal{F}_+$.
		Furthermore, from such $\mathcal{L}_+$, we can construct a desired laminar family $\mathcal{L}$
		by $\mathcal{L} := \{ X \cap [n] \mid X \in \mathcal{L}_+ \}$.
		Thus, by \cref{thm:laminarization},
		{\sc Full Laminarization} can be solved in $O(n^4)$ time.
	\end{proof}
	
In \cref{subsec:full algo},
we give an $O(rn^4)$-time algorithm for {\sc Full Displaying} (\cref{thm:algo 5}).
Thus, by \cref{thm:full laminarization,thm:algo 5},
we obtain \cref{thm:main} for full $\mathcal{A}$-partite quartet systems.

\subsection{Algorithm for full multipartite quartet system}\label{subsec:full algo}
Our proposed algorithm for full multipartite quartet systems is devised by combining Algorithm~4 for complete multipartite quartet systems and an algorithm for full quartet systems.
For full quartet system $\mathcal{Q}$,
it is known~\cite{AAM/BD86} that
\QCP\ can be solved in linear time of $|\mathcal{Q}|$,
and that a phylogenetic tree displaying $\mathcal{Q}$ is uniquely determined.
By summarizing these facts with notations introduced in this paper,
we obtain the following.
\begin{thm}[\cite{AAM/BD86,BJMSP/CS81}]\label{thm:full}
	Suppose that $\mathcal{Q}$ is full on $[n]$.
	Then \QCP\ can be solved in $O(|\mathcal{Q}|)$ time.
	Furthermore, if $\mathcal{Q}$ is compatible,
	then a weak $\{[n]\}$-cut family $\mathcal{F}$ displaying $\mathcal{Q}$ is uniquely determined up to $\sim$.
\end{thm}
	
Let $\mathcal{A} := \{A_1, A_2, \dots, A_r\}$ be a partition of $[n]$ with $|A_i| \geq 2$ for all $i \in [r]$.
Suppose that a full $\mathcal{A}$-partite quartet system $\mathcal{Q} =\mathcal{Q}_0 \cup \mathcal{Q}_1 \cup \cdots \cup \mathcal{Q}_r$ is compatible.
Then we can obtain a minimal laminarizable $\mathcal{A}$-cut family $\mathcal{F}_0$ displaying $\mathcal{Q}_0$ and a laminar weak $\mathcal{A}$-cut family $\mathcal{L}_i \subseteq 2^{A_i}$ displaying $\mathcal{Q}_i$ for each $i \in [r]$.
By combining $\mathcal{F}_0, \mathcal{L}_1, \dots, \mathcal{L}_r$ appropriately,
we can construct a minimal laminarizable weak $\mathcal{A}$-cut family displaying $\mathcal{Q}$ as follows.
\begin{description}
	\item[Algorithm~5 (for {\sc Full Displaying}):]
	\item[Input:] A full $\mathcal{A}$-partite quartet system $\mathcal{Q} = \mathcal{Q}_0 \cup \mathcal{Q}_1 \cup \cdots \cup \mathcal{Q}_r$.
	\item[Output:] 
	Either detect the incompatibility of $\mathcal{Q}$,
	or obtain a weak $\mathcal{A}$-cut family $\mathcal{F}$ displaying $\mathcal{Q}$.
	\item[Step 1:]
	Solve {\sc Displaying} for $\mathcal{Q}_0$ by Algorithm~4
	and \QCP\ for $\mathcal{Q}_i$ for $i \in [r]$.
	If algorithms detect the incompatibility of $\mathcal{Q}_i$ for some $i$,
	then output ``$\mathcal{Q}$ is not compatible'' and stop.
	Otherwise, obtain an $\mathcal{A}$-cut family $\mathcal{F}_0$ displaying $\mathcal{Q}_0$
	and a laminar weak $\mathcal{A}$-cut family $\mathcal{L}_i \subseteq 2^{A_i}$ displaying $\mathcal{Q}_i$ for each $i \in [r]$.
	\item[Step 2:] Let $\mathcal{F}_i := \{ X \cap A_i \mid X \in \mathcal{F}_0 \textrm{ with }\ave{X} \supseteq A_i,\ X \cap A_i \textrm{ is a weak $\mathcal{A}$-cut} \}$ for $i \in [r]$.
	If $\mathcal{F}_i / {\sim} \not\subseteq \mathcal{L}_i / {\sim}$,
	then output ``$\mathcal{Q}$ is not compatible'' and stop.
	\item[Step 3:] Define $\mathcal{F} := \mathcal{F}_0 \cup \bigcup_{i \in [r]} \{ Y \in \mathcal{L}_i \mid Y \not\sim X \textrm{ for all } X \in \mathcal{F}_i \}$.
	If $|\mathcal{F}| \leq 2n$,
	then output $\mathcal{F}$.
	Otherwise, output ``$\mathcal{Q}$ is not compatible.''
	\qqed
\end{description}

\begin{thm}\label{thm:algo 5}
	Algorithm~{\rm 5} solves {\sc Full Displaying} in $O(rn^4)$ time.
	Furthermore, if the input is compatible,
	then the output is a minimal laminarizable weak $\mathcal{A}$-cut family.
\end{thm}
	\begin{proof}
		(Validity).
		It suffices to show that (i) if Algorithm~5 reaches Step~3,
		then $\mathcal{F}$ defined in Step~3 displays $\mathcal{Q}$,
		and (ii) if $\mathcal{Q}$ is compatible,
		then Algorithm~5 outputs a laminarizable family $\mathcal{F}$.
		
		(i).
		Observe that a weak $\mathcal{A}$-cut family $\mathcal{F}$ displays $\mathcal{Q}$
		if and only if the family of $\mathcal{A}$-cuts in $\mathcal{F}$ displays $\mathcal{Q}_0$
		and $\{ X \cap A_i \mid X \in \mathcal{F} \textrm{ such that } \ave{X} \supseteq A_i,\
		X \cap A_i \textrm{ is a weak $\mathcal{A}$-cut} \} \sim \mathcal{L}_i$ for every $i \in [r]$.
		Note that the latter condition follows from \cref{thm:full}.
		Hence the output $\mathcal{F}$ of Algorithm~5 displays $\mathcal{Q}_i$ for all $i \in [r]$.
		Indeed, by the definition of $\mathcal{F}$, the family of $\mathcal{A}$-cuts in $\mathcal{F}$ is equal to $\mathcal{F}_0$,
		and $\{ X \cap A_i \mid X \in \mathcal{F} \textrm{ such that } \ave{X} \supseteq A_i,\
		X \cap A_i \textrm{ is a weak $\mathcal{A}$-cut} \} \sim \mathcal{L}_i$ for every $i \in [r]$.
		
		(ii).
		Suppose that $\mathcal{Q}$ is compatible.
		Let $\mathcal{F}^*$ be a laminar weak $\mathcal{A}$-cut family displaying $\mathcal{Q}$.
		First we prove that Algorithm~5 reaches Step~3.
		Since $\mathcal{Q}$ is compatible,
		so are $\mathcal{Q}_0, \mathcal{Q}_1, \dots, \mathcal{Q}_r$.
		Suppose, to the contrary, that $\mathcal{F}_i / {\sim} \not\subseteq \mathcal{L}_i / {\sim}$ for some $i \in [r]$.
		By \cref{prop:r>=3 unique} and \cref{thm:algo 4},
		it holds $\mathcal{F}_0 / {\sim} \subseteq \mathcal{F}^*_0 / {\sim}$,
		where $\mathcal{F}^*_0$ is the family of $\mathcal{A}$-cuts in $\mathcal{F}^*$.
		Then it follows from $\mathcal{F}_i / {\sim} \not\subseteq \mathcal{L}_i / {\sim}$ that $\{ X \cap A_i \mid X \in \mathcal{F}^* \textrm{ such that } \ave{X} \supseteq A_i,\
		X \cap A_i \textrm{ is a weak $\mathcal{A}$-cut} \} \not\sim \mathcal{L}_i$,
		contradicting that $\mathcal{F}^*$ displays $\mathcal{Q}_i$.
		
		Next we prove that $\mathcal{F}$ is laminarizable.
		It suffices to show that $\mathcal{F} / {\sim} \subseteq \mathcal{F}^* / {\sim}$.
		Suppose, to the contrary, that $\mathcal{F} / {\sim} \not\subseteq \mathcal{F}^* / {\sim}$.
		Then there is $X \in \mathcal{F}$ with $X \not\sim Y$ for all $Y \in \mathcal{F}^*$.
		By \cref{prop:r>=3 unique} and \cref{thm:algo 4},
		it holds that $X \in \mathcal{L}_i$ for some $i \in [r]$.
		Since $\{ X \cap A_i \mid X \in \mathcal{F}^* \textrm{ such that } \ave{X} \supseteq A_i,\
		X \cap A_i \textrm{ is a weak $\mathcal{A}$-cut} \} \sim \mathcal{L}_i$,
		there is $X^* \in \mathcal{F}^*$ such that $X^* \cap A_i \sim X$.
		If $X^* \in \mathcal{F}^* \setminus \mathcal{F}_0^*$,
		then $\ave{X^*} = A_i$.
		This implies $X^* \sim X$, contradicting the assumption of $X$.
		Suppose that $X^* \in \mathcal{F}_0^*$.
		Then $\ave{X^*} \supseteq A_i \cup A_j$ for some $j \in[r] \setminus \{i\}$.
		By the definition of $\mathcal{F}$,
		there is no $Y \in \mathcal{F}_0$ with $Y \cap A_i \sim X$.
		This implies that $(\mathcal{F}_0)_{ij} \not\sim (\mathcal{F}_0^*)_{ij}$,
		contradicting \cref{prop:r=2 unique}.
		
		The above proof implies that $\mathcal{F}$ is a minimal laminarizable weak $\mathcal{A}$-cut family.
		$|\mathcal{F}| \leq 2n$ follows from the laminarizability and minimality of $\mathcal{F}$.

		(Complexity).
		By \cref{thm:algo 4,thm:full},
		Step~1 can be done in $O(rn^4 + \sum_{i \in [r]} |\mathcal{Q}_i|) = O(rn^4)$ time.
		Steps~2 and~3 take $O(n^2)$ time.
		Thus the running-time of Algorithm~5 is $O(rn^4)$.
	\end{proof}
	
By the proof of \cref{thm:algo 5},
the following corollary holds.
\begin{cor}
	Suppose that a full $\mathcal{A}$-partite quartet system $\mathcal{Q}$ is compatible.
	Then a minimal laminarizable weak $\mathcal{A}$-cut family $\mathcal{F}$ displaying $\mathcal{Q}$ is uniquely determined up to $\sim$.
\end{cor}

\section{Concluding remark}\label{sec:remark}
We conclude this paper by suggesting a possible situation in which our results are applicable.
Quartet-based phylogenetic tree reconstruction methods
may be viewed as qualitative approximations of 
distance methods that construct a phylogenetic tree from 
(evolutionary) distance 
$\delta : [n] \times [n] \rightarrow \mathbf{R}_+$ among a set $[n]$ of taxa.
Here $\mathbf{R}_+$ denotes the set of nonnegative real values.
The distance $\delta$ naturally gives rise to a full quartet system 
${\cal Q}$ as follows. 
Let ${\cal Q} := \emptyset$ at first. 
For all distinct $a,b,c,d \in [n]$,
add $ab||cd$ to ${\cal Q}$ if 
$\delta(a,b) + \delta(c,d) < \min \{ \delta(a,c) + \delta(b,d), \delta(a,d) + \delta(b,c)\}$.
See~\cite{JME/F81, Psyco/ST77}.
Then ${\cal Q}$ becomes a full quartet system, 
by adding $ab|cd$, $ac|bd$, $ac|bd$
if none of $ab||cd$, $ac||bd$, $ac||bd$ belong to ${\cal Q}$.
If $\delta$ coincides with the path-metric of an actual phylogenetic tree $T$ 
(with nonnegative edge-length),  
then $\delta$ obeys the famous four-point condition on all four elements $a,b,c,d$~\cite{MAHS/B71}:
\begin{description}
	\item[(4pt)] the larger two of $\delta(a,b) + \delta(c,d)$, $\delta(a,c) + \delta(b,d)$, and $\delta(a,d) + \delta(b,c)$ are equal.
\end{description}
In this case, the above definition of quartets 
matches the neighbors relation of $T$.
Thus,
from the full quartet system ${\cal Q}$, via the algorithm of~\cite{AAM/BD86},
we can recover the original phylogenetic tree $T$ (without edge-length). 

Next we consider the following limited situation in which complete/full $\mathcal{A}$-partite quartet systems naturally arise.
The set $[n]$ of taxa is divided into $r$ groups $A_1, A_2, \dots, A_r$ (with $|A_i| \geq 2$).
By reasons of the cost and/or the difficulty of experiments, 
we are limited to measure the distance between $a \in A_i$ and $b \in A_j$ via different methods/equipments depending on $i, j$.
Namely we have $\binom{r}{2}$ distance functions $\delta_{ij} : A_i \times A_j \rightarrow \mathbf{R}_+$ for $1 \leq i < j \leq r$ but
it is meaningless to compare numerical values 
of $\delta_{ij}$ and $\delta_{i'j'}$ for $\{i, j\} \neq \{i',j'\}$.
A complete ${\cal A}$-partite quartet system ${\cal Q}$ is obtained as follows.
For distinct $i,j$, define complete bipartite quartet system ${\cal Q}_{ij}$ by:
for all distinct $a,a' \in A_i$ and $b,b' \in A_j$ it holds
\begin{align*}
ab || a'b' \in {\cal Q}_{ij} \quad &\mbox{if}\ \delta_{ij}(a, b) + \delta_{ij}(a', b') < \delta_{ij}(a, b') + \delta_{ij}(a', b),\\
ab' || a'b \in {\cal Q}_{ij} \quad &\mbox{if}\  \delta_{ij}(a, b) + \delta_{ij}(a', b') > \delta_{ij}(a, b') + \delta_{ij}(a', b),\\
aa' | bb' \in {\cal Q}_{ij}   \quad  &\mbox{if}\  \delta_{ij}(a, b) + \delta_{ij}(a', b') = \delta_{ij}(a, b') + \delta_{ij}(a', b).
\end{align*}
Then ${\cal Q} := \bigcup_{1 \leq i < j \leq r} {\cal Q}_{ij}$ 
is a complete ${\cal A}$-partite quartet system. 

This construction of complete ${\cal A}$-partite quartet system ${\cal Q}$ 
is justified as follows.
Assume a phylogenetic tree $T$ on $[n]$ with path-metric $\delta$.
Assume further that each $\delta_{ij}$ is linear on $\delta$, i.e., 
$\delta_{ij}$ is equal to $\alpha_{ij} \delta$ 
for some unknown constant $\alpha_{ij} > 0$.
By (4pt), the situation $\delta_{ij}(a, b) + \delta_{ij}(a', b') < \delta_{ij}(a, b') + \delta_{ij}(a', b)$
implies $\delta(a,b) + \delta(a',b') < \delta(a, b') + \delta(a', b) = \delta(a,a') + \delta(a,b')$, 
and implies that $T$ displays $ab||a'b'$. 
The situation  $\delta_{ij}(a, b) + \delta_{ij}(a', b') = \delta_{ij}(a, b') + \delta_{ij}(a', b)$
implies $\delta(a,b) + \delta(a',b') = \delta(a, b') + \delta(a', b) \geq \delta(a,a') + \delta(a,b')$, and implies that $T$ displays $aa' | bb'$.
Thus, by our algorithm, we can construct 
a phylogenetic tree $T'$ ``similar'' to $T$ in the sense that 
$T'$ and $T$ produce the same result under 
our limited measurement.

Suppose now that we have additional $r$ distance functions $\delta_{i} : A_i \times A_i \rightarrow \mathbf{R}_+$ for $i \in [r]$.
In this case, we naturally obtain a full $\mathcal{A}$-partite quartet system.
Indeed, define full quartet system ${\cal Q}_{i}$ on $A_i$
according to $\delta_{i}$ as in the first paragraph.
Then  $\mathcal{Q} := \bigcup_{1 \leq i < j \leq r} {\cal Q}_{ij} \cup \bigcup_{1 \leq i \leq r} \mathcal{Q}_i$
is a full $\mathcal{A}$-partite quartet system 
to which our algorithm is applicable.

	\section*{Acknowledgments}
	We thank Kunihiko Sadakane for bibliographical information.
	The first author is supported by JSPS KAKENHI Grant Numbers JP26280004, JP17K00029.
	The second author is supported by JSPS KAKENHI Grant Numbers JP16J04545, JP17K00029, JP19J01302, 20K23323, 20H05795, Japan.
	
	This version of the article has been accepted for publication, after peer review (when applicable) but is not the Version of Record and does not reflect post-acceptance improvements, or any corrections. The Version of Record is available online at: http://dx.doi.org/10.1007/s00453-022-00945-9.
	Use of this Accepted Version is subject to the publisher’s Accepted Manuscript terms of use https://www.springernature.com/gp/open-research/policies/accepted-manuscript-terms.

\end{document}